\documentclass[11 pt]{amsart}
\usepackage{amsfonts}

\usepackage{amssymb, amsmath}
\usepackage{natbib}
\usepackage{lscape}
\usepackage{dsfont}
\usepackage{mathrsfs}
\usepackage{bbm}
\usepackage[pdftex,plainpages=false,colorlinks,hyperindex,bookmarksopen,linkcolor=red,citecolor=blue,urlcolor=blue]{hyperref}
\usepackage{graphicx}
\usepackage{subcaption}

\DeclareMathAlphabet{\mathpzc}{OT1}{pzc}{m}{it}
\bibpunct{[}{]}{;}{n}{,}{,}
\newtheorem{te}{Theorem}[section]
\newtheorem{defin}[te]{Definition}
\newtheorem{os}[te]{Remark}
\newtheorem{prop}[te]{Proposition}

\newtheorem{ex}[te]{Example}
\numberwithin{equation}{section}
\allowdisplaybreaks

\begin{document}

\large

\title[]{ Some families of random fields  related to multiparameter L\'evy
processes }
\keywords{Multiparameter L\'evy processes, subordination of random fields, 
fractional operators, semi-Markov processes, anomalous diffusion\\
\\
* Dipartimento di Scienze di Base Applicate all'Ingegneria - Sapienza Università di Roma
\\
** Dipartimento di Scienze Statistiche - Sapienza Università di Roma
\\
Corresponding author: Costantino Ricciuti (costantino.ricciuti@uniroma1.it)
}
\date{\today }

 \author[]{Francesco Iafrate *   }
 \author[]{Costantino Ricciuti**}

\begin{abstract}
Let $\mathbb{R}^N_+= [0,\infty)^N$. We here consider a class of random fields $(X_t)_{t\in \mathbb{R}^N_+}$ which are known as Multiparameter L\'evy processes. Related multiparameter semigroups of operators  and their generators are represented as pseudo-differential operators. We also consider the composition of $(X_t)_{t\in \mathbb{R}^N_+}$ by means of the so-called subordinator fields and we provide a Phillips formula. We finally study the composition of $(X_t)_{t\in \mathbb{R}^N_+}$ by means of the so-called inverse random fields, which gives rise to interesting long range dependence properties. As a byproduct of our analysis, we study a model of anomalous diffusion in an anisotropic medium which extends the one treated in \cite{macci}.

\end{abstract}

\maketitle

{}

{}

\section{Introduction}
In this paper we consider Multiparameter L\'evy processes $(X_t)_{t\in \mathbb{R}^N_+}$ in the sense of \cite{sato, pedersen 1, pedersen 2, pedersen 3}.  The reason they are called in this way is that they enjoy, in some sense,  independence and stationarity of increments.  Independence  of increments is meant in the following way. First a partial ordering on $\mathbb{R}^N_+$ is established, such that $a\preceq b$ in $\mathbb{R}^N_+$ if $a_i\leq b_i$ for each $i=1, \dots N$. Then it is assumed that, for any choice of ordered points $t^{(1)}, t^{(2)}, \dots, t^{(k)}$ in $\mathbb{R}^N_+$, we have that $X_{t^{(j+1)}}-X_{t^{(j)}}$, $j=1, \dots, k-1$, is a set of independent random variables. On the other hand, stationarity of increments means that  $X_{t+\tau}-X_t$ has the same distribution of $X_\tau$  for all $t, \tau \in \mathbb{R}^N_+$.

Such processes are not to be confused with other extensions of L\'evy processes where the parameter is multidimensional. Among them, we recall a class of processes, including the Brownian sheet and  the Poisson sheet, which have a different definition from ours, because in that case independence of increments is understood in another way (consult e.g. \cite{adler,dalang, khosh 2}).


Multiparameter L\'evy processes are of interest in Analysis since they furnish a stochastic solution to some systems of differential equations, as will be recalled in  section \ref{sezione 2}. Roughly speaking,   if the vector $G= (G_1, G_2, \dots, G_N)$ is the generator of a Multiparameter L\'evy process $(X_t)_{t\in \mathbb{R}^N_+}$ , then, provided that $u$ belongs to suitable function spaces, the function $\mathbb{E}u (x+X_t)$ ($\mathbb{E}$ denoting the expectation)   solves the system 
\begin{align}
\frac{\partial}{\partial t_k} h(x,t)= G_k \, h(x,t) \qquad h(x,0)= u(x) \qquad k=1, \dots, N \label{sistema 1}
\end{align}
where $t=(t_1, \dots, t_N)$.
Of course, for one parameter L\'evy processes, we have a single differential equation, as stated by the well known Feller theory of one parameter Markov processes and semigroups.

The idea of subordination for Multiparameter L\'evy processes is presented in \cite{sato, pedersen 1, pedersen 2, pedersen 3} (for the classical theory of subordination of one-parameter L\'evy processes see e.g. [\cite{sato book}, chapter 6]). The construction is as follows. Let $(X_t)_{t\in \mathbb{R}_+ ^N}$ be a Multiparameter L\'evy process and let $(H_t)_{t\in \mathbb{R}_+ ^M}$ be a  subordinator field, i.e. a Multiparameter L\'evy process with values in $\mathbb{R}_+^N$, such that it has non decreasing paths in the sense of the partial ordering (i.e. $t_1 \preceq t_2$ in $\mathbb{R}^M_+$ implies $H_{t_1}\preceq H_{t_2}$ in $\mathbb{R}^N_+$) . The subordinated field  is defined by $(X_{H_t})_{t\in  \mathbb{R}_+ ^M}$ and it is again a Multiparameter L\'evy process.

 One of the main results of this paper is to provide a formula for the generator of the subordinated field. Indeed  we find an  extension of the Phillips theorem to the multi-parameter case, by involving the so-called multi-dimensional Bernstein functions. This gives rise to interesting systems of type \ref{sistema 1}. In those systems, the operator on the right side may possibly be pseudo-differential.
For example, when the subordinator field is stable,  such a system could be interesting for those studying fractional equations, since the operator on the right side involves the fractional Laplacian and the so-called fractional gradient; we recall that the fractional gradient is a generalization of the  fractional Laplacian to the case where the jumps are not isotropically distributed (see e.g. [\cite{macci}, Example 2.2] and the references therein).

The basic case of subordinator field is the one with $M=1$. In this case we have a one-parameter process $H_t= (H_1 (t), \dots, H_N(t))$ which the authors in \cite{sato} call \textit{multivariate subordinator}. This is nothing more than a one-parameter L\'evy process with values in $\mathbb{R}^N_+$, where all the components $t\to H_j(t)$ are non-decreasing (namely, each $H_j$ is a subordinator). Using a multivariate subordinator, subordination of a Multiparameter L\'evy process gives a one-parameter L\'evy process.

In the second part of the paper, by considering a multivariate subordinator  $(H_1(t), \dots H_N(t))$, we will construct  a new random field 
\begin{align}\mathcal{L}_t= \bigl (L_1(t_1), \dots, L_N(t_N) \bigr ) \qquad t=(t_1, \dots, t_N)\label{inverse random field 1} \end{align}
where  $L_j$ is the inverse, also said the hitting time, of the subordinator $H_j$, i.e.
$$ L_j(t_j)= \inf \{x>0: H_j(x)>t_j \}.$$
We will call \ref{inverse random field 1} \textit{inverse random field}.
 Now, let $(X_t)_{t\in \mathbb{R}^N_+}$ be a Multiparameter L\'evy process with values in $\mathbb{R}^d$, which is assumed to be independent of  (\ref{inverse random field 1}) . We are interested in the subordinated random field $(Z_t)_{t\in \mathbb{R}^N_+}$ defined by
\begin{align}
Z_t=X_{\mathcal{L}_t} \label{campo subordinato con gli inversi 1}
\end{align}
Of course,  \ref{inverse random field 1} and  \ref{campo subordinato con gli inversi 1}    are not Multiparameter L\'evy processes because they enjoy neither independence nor stationarity of increments with respect to the partial ordering on $\mathbb{R}^N_+$.  However, they may be useful in applications in order to model spatial data exhibiting various correlation structures which cannot fall in the framework of Multiparameter L\'evy or Markov processes.

Our topic has been inspired by some existing literature. First of all, there are many papers  (see e.g. \cite{becker,
kolokoltsov,
meer nane,
meer spa,
meer jap,
meer straka,
meer libro,
meer toaldo,
straka 2})
  concerning semi-Markov processes of the form
\begin{align}
Z(t)= X(L(t)) \qquad t\geq 0 \label{semi Markov iniziale}
\end{align}
where $X$ is a (one parameter) L\'evy process in $\mathbb{R}^d$ and $L$ is the inverse  of a subordinator $H$ independent of $X$,  i.e.
$$ L(t)= \inf \{x>0: H(x)>t\}. $$
Processes of  type \ref{semi Markov iniziale} have an important role in statistical physics, since they model continuous time random walk scaling limits and anomalous diffusions. Moreover, it is known that \ref{semi Markov iniziale} is not Markovian and its density $p(x,t)$ is governed by an equation which is non local in the time variable:
\begin{align}\mathcal{D}_t p(x,t)  - \overline{\nu} (t) p(x,0)    = G^*\, p(x,t). \end{align}
In the above equation, $G^*$ is the dual to the generator of $X$ and the operator $\mathcal{D}_t$ is the so-called generalized fractional  derivative (in the sense of Marchaud), defined by
\begin{align}
\mathcal{D}_t h(t):= \int _0^\infty \bigl ( h(t)- h(t-\tau) \bigr ) \nu (d\tau), \qquad t>0 \label{prima equazione}
\end{align}
where $\nu$ is the L\'evy measure of $H$ and $\overline{\nu}(t):= \int _t^\infty \nu (dx)$ is the tail of the L\'evy measure.

The main results regarding the random fields of type \ref{campo subordinato con gli inversi 1} will be reported   in Section \ref{paragrafo sugli inversi}; we will show that they have interesting correlation structures and that they are governed by particular integro-differential  equations. Such equations are non local in the $t_1, \dots, t_N$ variables and generalize equation (\ref{prima equazione}) holding in the one-parameter case. 

We also recall that the first idea of inverse random field appeared in [\cite{macci}, Sect. 3] where the authors proposed a model of multivariate time change.

Another source of inspiration is the paper \cite{Leonenko}, even if it does not exactly fit into our context. Here the authors considered a Poisson sheet $N(t_1, t_2)$, which is not a Multiparameter L\'evy process in the sense of this paper, and studied the composition 
$$Z(t_1, t_2)=N(L_1(t_1),L_2( t_2)),$$
 where $L_1$ and $L_2$ are  two independent  inverse stable subordinators, of index $\alpha _1$ and $\alpha _2$ respectively; the resulting random field showed interesting long range dependence properties.



\section{Basic notions and some preliminary results} \label{sezione 2}
We introduce the partial ordering on the set  $\mathbb{R}_+^N= [0,\infty)^N$: the point $a=(a_1, \dots, a_N)$ precedes the point  $b=(b_1, \dots, b_N)$, say  $a \preceq b$, if and only if $a_j \leq b_j$ for each $j=1, \dots, N$. 

A sequence $\{x_i\}_{i=1}^\infty$ in $\mathbb{R}_+^N$ is said to be increasing if $x_i \preceq x_{i+1}$ for each $i$; it  is said to be decreasing if  $x_{i+1} \preceq x_i$ for each $i$.

Consider a  function  $f: \mathbb{R}_+^N\to \mathbb{R}^d$. We say that $f$ is right continuous at $x \in \mathbb{R}_+^N$ if, for any decreasing sequence  $x_i\to x$ we have $f(x_i) \to f(x)$.

We say that  $f: \mathbb{R}_+^N\to \mathbb{R}^d$ has left limits at $x \in \mathbb{R}_+^N/\{0\}$ if, for any increasing sequence  $x_i\to x$, the limit of $f(x_i)$ exists; such a limit may depend on the choice of the sequence $x_i$.

Moreover, $f$ is said to be cadlag if it is right continuous at each $x \in \mathbb{R}_+^N$ and has left limits at each  $x \in \mathbb{R}_+^N/\{0\}$.

\subsection{Multiparameter L\'evy processes}

We here recall the notion of Multiparameter L\'evy process in the sense of \cite{sato, pedersen 1, pedersen 2, pedersen 3}. We also refer to \cite{khosh} as a standard reference on Multiparameter Markov processes. 

The parameters set is  here assumed to be $\mathbb{R}_+^N$.  An analogous (but more general) definition holds if  the parameter set is  any cone contained in  $\mathbb{R}^N$, but this generalization is not essential for the aim of this paper.

\begin{defin} \label{Multiparameter Levy processes}
A random field $(X_t)_{t\in \mathbb{R}_+^N}$, with values in $\mathbb{R}^d$, is said to be a Multiparameter L\'evy process if
\begin{enumerate}
\item  $X_0=0$ a.s.
\item it has independent increments with respect to the partial ordering on  $\mathbb{R}_+^N$, i.e. for any choice of $ 0 =t^{(0)} \preceq t^{(1)} \preceq t^{(2)} \dots \preceq t^{(k)}$, the random variables $X_{t^{(j)}}-X_{t^{(j-1)}}$, $j=1,\dots, k$, are independent.
\item it has stationary increments, i.e. $X_{t+\tau}-X_t\overset{d}{=}X_\tau$ for each $t,\tau \in \mathbb{R}^N_+$
\item it is cadlag a.s.
\item it is continuous in probability, namely, for any sequence $t^{(i)} \in \mathbb{R}_+^N$ such that $t^{(i)}\to t$, it holds that $X_{t^{(i)}}$ converges to $X_t$ in probability. 
\end{enumerate}
\end{defin}
If $(1),(2), (3), (5)$ hold, then $(X_t)_{t\in \mathbb{R}_+^N}$ is said to be a Multiparameter L\'evy process in law.

We report some examples of Multiparameter L\'evy processes, which are constructed from one-parameter ones.

\begin{ex} \label{processi additivi}If  $(X^{(1)}_{t_1})_{t_1 \in \mathbb{R}_+}, \dots, (X^{(N)}_{t_N})_{t_N \in \mathbb{R}_+}$ are $N$ independent L\'evy processes on $\mathbb{R}^d$, with laws $\nu ^{(1)} _{t_1}, \dots, \nu ^{(N)} _{t_N}$, then 
$$
X_t:= X^{(1)}_{t_1}+X^{(2)}_{t_2}+ \dots +   X^{(N)}_{t_N}     \qquad t=(t_1, t_2, \dots, t_N)
$$
is a N-parameter L\'evy process on $\mathbb{R}^d$, which is usually called additive L\'evy process (see e.g. \cite{khosh1} and [\cite{khosh}, pp. 405]).

Here $X_t$ has law
$$ \mu _t = \nu ^{(1)} _{t_1} *\dots* \nu ^{(N)} _{t_N}
$$
where $*$ denotes the convolution.
Examples of the sample paths are shown in \autoref{fig:add-field} and \autoref{fig:vec-field}.

\begin{figure}[h]
	\begin{subfigure}{.5\textwidth}
		\includegraphics[width=\textwidth]{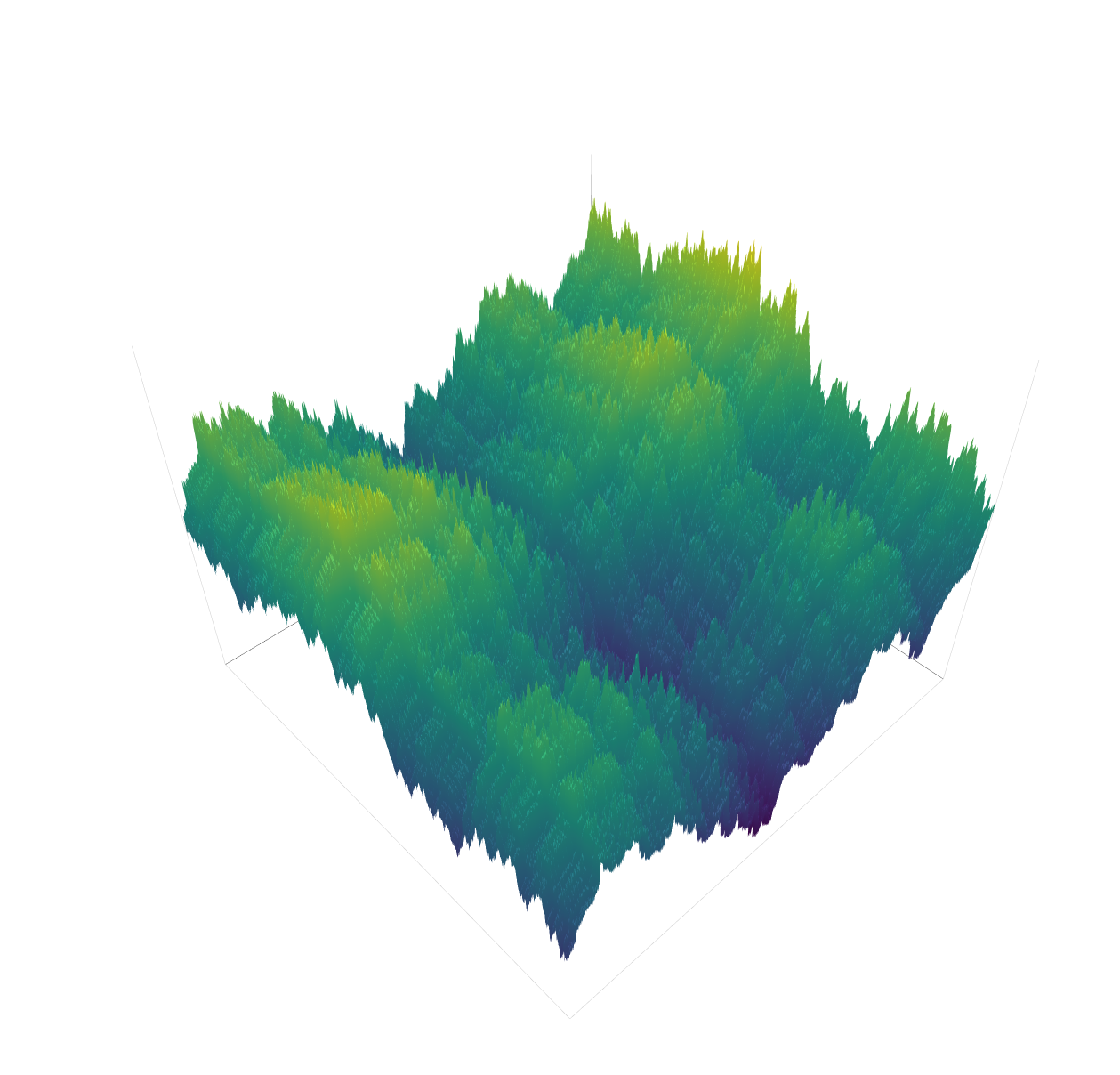}
		\caption{Brownian additive field}		
	\end{subfigure}%
	\begin{subfigure}{.5\textwidth}
		\includegraphics[width=\textwidth]{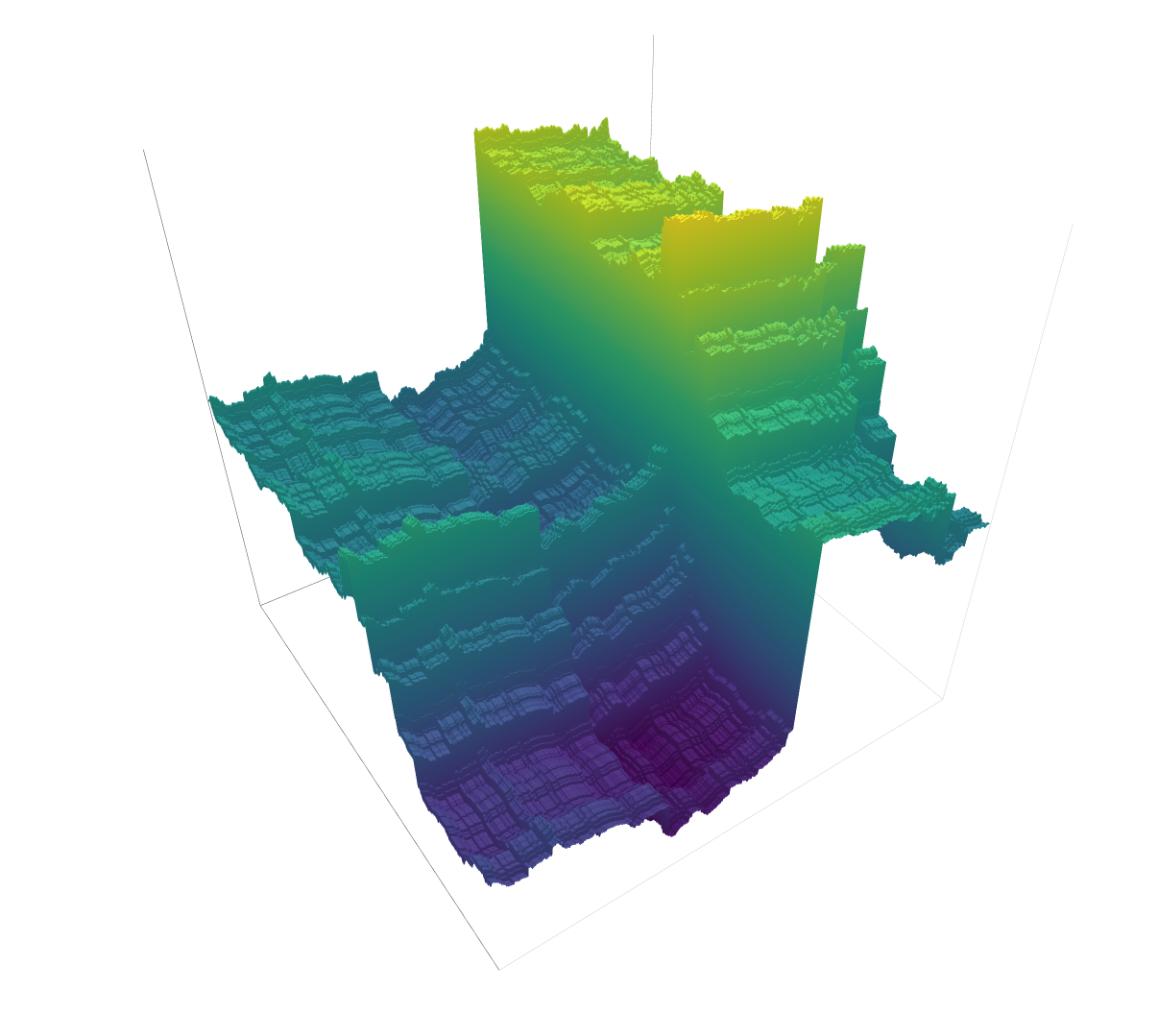}
		\caption{Stable additive field}	
	\end{subfigure}
\caption{Sample paths of additive Lévy fields, as in Example 2.2}
\label{fig:add-field}
\end{figure}

\begin{figure}[h]	
	\begin{subfigure}{.5\textwidth}
		\centering
		\includegraphics[width=.9\textwidth]{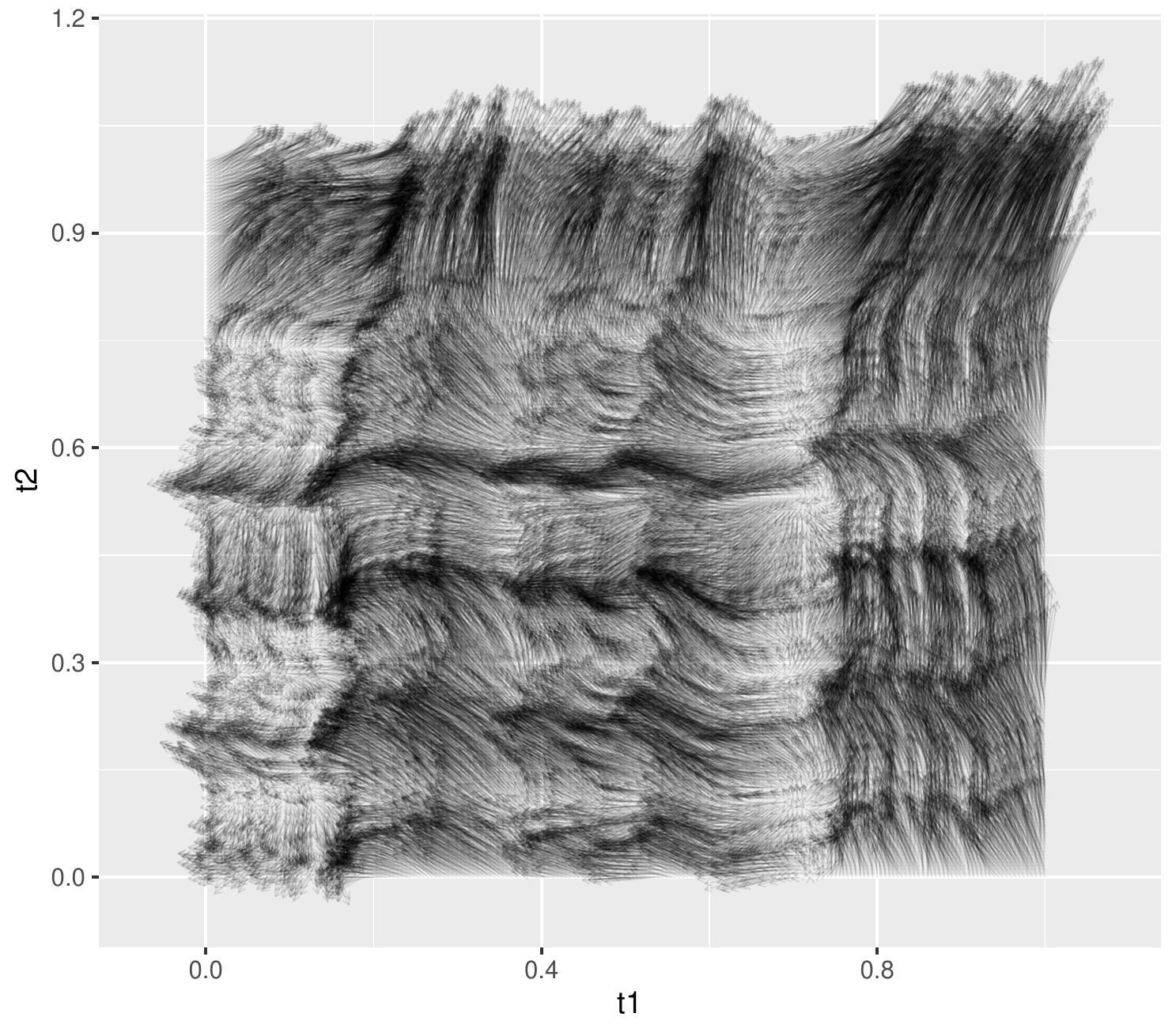}
		\caption{Brownian vector field}		
	\end{subfigure}%
	\begin{subfigure}{.5\textwidth}
		\centering
		\includegraphics[width=.9\textwidth]{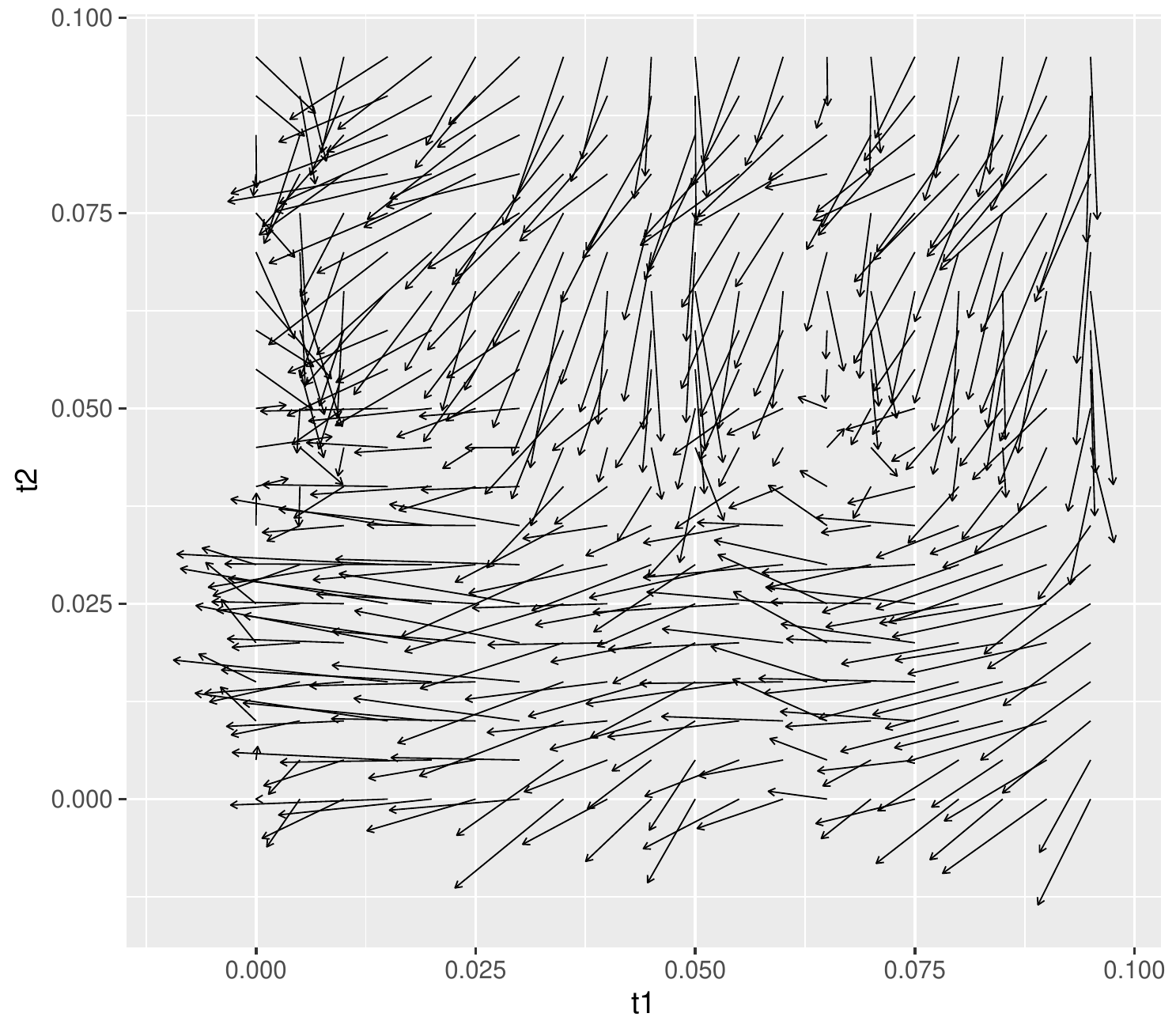}
		\caption{Zoom of the previous figure on $[0,1]^2$}	
	\end{subfigure}
	\caption{Sample path of a $ \mathbb R^2 $-valued biparameter additive field (i.e. $ d = N = 2 $).}
	\label{fig:vec-field}
\end{figure}

\end{ex}

\begin{ex}
Let $(X^{(1)}_{t_1})_{t_1 \in \mathbb{R}_+}, \dots (X^{(N)}_{t_N})_{t_N \in \mathbb{R}_+}$ be independent $\mathbb{R}$-valued L\'evy processes with laws    $\nu ^{(1)}_{t_1}, \dots, \nu ^{(N)} _{t_N}$. Then 
$$ X_t = \bigl ( X^{(1)}_{t_1},   X^{(2)}_{t_2},    \dots, X^{(N)}_{t_N} \bigr )   \qquad t=(t_1, t_2, \dots, t_N)$$
is a $\mathbb{R}^N$ valued L\'evy process, which can be called product L\'evy process (in the language of [\cite{khosh}, pag. 407]). Clearly, this is a particular case of  Example \ref{processi additivi} because
$$ X_t= X^{(1)}_{t_1}e_1+ X^{(2)}_{t_2}e_2 + \dots + X^{(N)}_{t_N}e_N $$
where $\{e_1, \dots, e_N\}$ denotes the canonical basis of $\mathbb{R}^N$.

Here $X_t$ has law
$$
\mu _t = \nu ^{(1)} _{t_1} \otimes    \nu ^{(2)} _{t_2} \dots \otimes \nu ^{(N)}_{t_N}
$$
where $\otimes$ denotes the product of measures.

\end{ex}

\begin{ex}
Let $(V_t)_{t\in \mathbb{\mathbb{R}_+}}$ be a L\'evy process in $\mathbb{R}^d$. Then $V_{c_1t_1+\dots+ c_Nt_N}$ is a multi-parameter L\'evy process for any choice of $ (c_1, \dots, c_N)\in \mathbb{R}^N_+$.
\end{ex}

\begin{os}  What we have presented is not the only way to extend the notion of independence of increments to the multiparameter case. A very common approach is to define independence of increments over disjoint rectangles
(see \cite{adler} and \cite{dalang}).
This gives rise to a class of random fields, known as Levy sheets (e.g. the Poisson sheet or the Brownian sheet). 
\end{os}


In the following, $\delta _0$ will denote the probability measure concentrated at the origin. Moreover, $\{e_1, \dots, e_N\}$ will denote the canonical basis of $\mathbb{R}^N$. 

\begin{defin} \label{semigruppo di misure}
A family $(\mu _t)_{t\in \mathbb{R}^N _+}$ of probability measures on $\mathbb{R}^d$ is said to be a $\mathbb{R}^N_+$-parameter convolution semigroup if 

i)  $\mu _{t+\tau}= \mu _t * \mu _{\tau}$, for all $t, \tau \in  \mathbb{R}^N_+$

ii)  $ \mu _t \to \delta _0 $ as $ t\to 0$
\end{defin}
By Def. \ref{semigruppo di misure} it follows that $\mu _t$ is infinitely divisible for each $t$. 

The above notion of multi-parameter convolution semigroup  is related to Multiparameter L\'evy processes, as shown in the following Proposition.

We preliminarily observe that, since $X_t$ is a Multiparameter L\'evy process, where $t=(t_1, \dots, t_N)$, it immediately follows that, for each $j=1, \dots, N$, the process  $(X_{t_je_j})_{t_j \in \mathbb{R}_+}$ is a classical one-parameter L\'evy process. In other words, if  $(\mu _t)_{t\in \mathbb{R}^N_+}$  is a multi-parameter convolution semigroup, then $(\mu _{t_je_j})_{t_j\in \mathbb{R}_+}$ is a one-parameter convolution semigroup which is the law of $X_{t_je_j}$.

\begin{prop} \label{multiparameter}
Let $(X_t)_{t\in \mathbb{R}_+^N}$ be a Multiparameter L\'evy process on $\mathbb{R}^d$ and let $\mu _t$ be the law of the random variable $X_t$. Then

\textit{i)} The family $(\mu _{t}) _{t  \in \mathbb{R}_+^N  }$ is a $\mathbb{R}^N_+$-parameter convolution semigroup of probability measures.

\textit{ii)}
There exist  independent random vectors $Y^{(j)}_{t_j}$, $j=1, \dots, N$,  with $ Y^{(j)}_{t_j} \overset{d}{=} X_{t_je_j}$, such that
\begin{align*}
X_t\overset{d}{=} Y^{(1)}_{t_1} + \dots Y^{(N)}_{t_N} \qquad t=(t_1, \dots, t_N)
\end{align*}

\end{prop}

\begin{proof}
By writing
$$X_{t+\tau}= (X_{t+\tau}-X_\tau) +X_\tau \qquad \textrm{for all}\,\, t,\tau \in \mathbb{R}_+^N$$
we observe that  $X_{t+\tau}-X_\tau$ and $X_\tau$ are independent by the assumption of independence of increments along those sequences that are increasing with respect to the partial ordering. Moreover $X_{t+\tau}-X_\tau$ has the same distribution of $X_t$ by stationarity. 
Hence  $\mu _{t+\tau}=\mu_t*\mu _\tau$. 
Moreover, stochastic continuity of $(X_t)_{t\in \mathbb{R}^N_+}$ gives $\mu_t \to \delta _0$ as $t\to 0$, and thus \textit{i)} is proved. To prove \textit{ii)}, it is sufficient to write $t= t_1e_1+ \dots+t_N e_N$ and apply the semigroup property just  proved in point \textit{i)}, to have
$$\mu _t = \mu _{t_1e_1}*\dots * \mu _{t_N e_N} $$
and the proof is complete since $\mu _{t_je_j}$ is the law of $X_{t_je_j}$.
\end{proof}

We stress that Proposition \ref{multiparameter} is a statement about equality in law of random variables ($t$ is fixed), and not equality of processes.

We further observe that Proposition \ref{multiparameter} says that to each Multiparameter L\'evy process in law there corresponds a unique convolution semigroup of probabilty measures. But, unlike what happens for classical L\'evy processes (i.e. when $N=1$),  the converse is not true in general: a multiparameter convolution semigroup $(\mu _t)_{t\in \mathbb{R}^N_+}$  can be associated to different Multiparameter L\'evy processes in law, because $(\mu _t)_{t\in \mathbb{R}^N_+}$ does not completely determine all the finite-dimensional distributions. Indeed, only along $\mathbb{R}^N_+$-increasing sequences $0 \preceq \tau ^{(1)} \preceq \dots \preceq \tau ^{(k)}$,  the joint distribution of $(X_{\tau^{(1)}}, \dots, X_{\tau ^{(k)}})$ can be uniquely determined in terms of $\mu _t$  by using independence and stationarity of increments, but this is not possible if the points $\tau ^{(1)}, \dots, \tau ^{(k)} \in \mathbb{R}^N_+$ are not ordered (in the sense of the partial ordering).

\subsubsection{Characteristic function of Multiparameter L\'evy processes} \label{MMMM}

Consider the $Y^{(j)}_{t_j}$ involved in Proposition \ref{multiparameter}. By the L\'evy Khintchine formula, we have
\begin{align}
\mathbb{E}e^{i\xi \cdot Y^{(j)}_{t_j} }= \int _{\mathbb{R}^d} e^{i\xi \cdot y} \mu _{t_je_j}(dy)=      e^{t_j\psi _j(\xi)}     \qquad \xi \in \mathbb{R}^d,  \label{LLLL}
\end{align}
 the L\'evy exponent $\psi _j$ having the form
\begin{align}
\psi _j (\xi)= i \gamma _j\cdot \xi - \frac{1}{2} A_j\xi \cdot \xi + \int _{\mathbb{R}^d/\{0\}} (e^{i \xi \cdot z}-1 -i\xi \cdot z I_{[-1,1]} (z)) \nu _j(dz)  \label{levy khintchine}
\end{align}
where $\gamma _j \in \mathbb{R}^d$, $A_j$ is the Gaussian covariance matrix,  $\nu _j$ denotes the L\'evy measure and $\cdot$ denotes the scalar product. 
By the above considerations, we thus get the following statement.

\begin{prop}
Let $(X_t)_{t\in \mathbb{R}_+^N}$ be a Multiparameter L\'evy process with values in $\mathbb{R}^d$. Then $X_t$ has characteristic function
\begin{align}
\mathbb{E} e^{i\xi\cdot X_t }= e^{t_1\psi _1(\xi)+ \dots + t_N \psi _N(\xi)} = e^{t\cdot \Psi (\xi)} \qquad \xi \in \mathbb{R}^d \label{funzione caratteristica Levy multiparametrico}
\end{align}
where $t= (t_1, \dots, t_N)$,  the functions $\psi _j$ have been defined in \ref{levy khintchine}, and
\begin{align}
\Psi (\xi )= (\psi_1(\xi), \dots, \psi _N(\xi)). \label{simbolo Levy multiparametrico}
\end{align}

\end{prop}

We will call \ref{simbolo Levy multiparametrico} the multidimensional L\'evy exponent.

\subsection{Autocorrelation function of Multiparameter L\'evy processes} \label{paragrafo autocorrelazione}

Consider a Multiparameter L\'evy process $\{X_t\}_{t\in \mathbb{R}_+^N}$ with values in $\mathbb{R}$. In the following Proposition we will explicitly compute the autocorrelation function between two ordered points in the parameter space, i.e. 
\begin{align}
\rho (X_s,X_t):= \frac{Cov (X_s, X_t)}{\sqrt{Var X_s} \sqrt{Var X_t}} \qquad s\preceq t .\label {autocorrelazione}
\end{align}
Of course, \ref{autocorrelazione} exists finite only in some cases, which will be specified in the following.
What we will find is the $N$-parameter extension of the well known formula holding in the case $N=1$, i.e. for classical L\'evy processes  (consult e.g. Remark 2.1 in \cite{Leonenko 2}):
\begin{align*}
\rho (X_s,X_t)= \sqrt{\frac{s}{t}} \qquad s\leq t.
\end{align*}

\begin{prop} \label{calcolo autocorrelazione Levy}
Let $\{X_t\}_{t\in \mathbb{R}_+^N}$ be a $N$-parameter L\'evy process with values in $\mathbb{R}$, having multidimensional L\'evy  exponent $\Psi (\xi)$     defined in \ref{funzione caratteristica Levy multiparametrico} and \ref{simbolo Levy multiparametrico}.
    For each $j=1, \dots, N$, let $\xi \to \psi _j(\xi)$ be twice differentiable  in a neighborhood of $\xi =0$, and such that $\psi'' _j(0)\neq 0$. Then the auto-correlation function defined in  \ref{autocorrelazione}    reads
\begin{align}
\rho (X_s,X_t)= \sqrt{ \frac{s\cdot \sigma ^2}{t\cdot \sigma ^2} } \qquad s\preceq t \label{autocorrelazione campi}
\end{align}
where $\cdot$ denotes the scalar product and $\sigma ^2 := -\Psi '' (0)$.
\end{prop}

\begin{proof} Consider the decomposition of $X_t$ given in Proposition \ref{multiparameter}.
Since $\psi _j'' (0)$ exists, then $Y_{t_j} ^{(j)}$ has finite mean and variance:
\begin{align*}
\mathbb{E} Y_{t_j}^{(j)} = -i t_j \psi _j'(0) = t_j\, \mathbb{E} Y_1^{(j)} \\ \mathbb{E} (Y_{t_j}^{(j)})^2 = -t_j   \psi _j''(0) -t_j^2 \psi _j '(0) ^2 \\ \mathbb{V}ar Y_{t_j}^{(j)} =  -t_j   \psi _j''(0) = t_j \mathbb{V}ar  Y_{1}^{(j)}
\end{align*}
Letting $\mu:= (\mathbb{E} Y_{1}^{(1)}, \dots, \mathbb{E} Y_{1}^{(N)} )$ and $ \sigma  ^2 := -\Psi '' (0)=  (\mathbb{V}ar  Y_{1}^{(1)}, \dots , \mathbb{V}ar Y_{1}^{(N)})$, we get
\begin{align*}
\mathbb{E}X_t= -it\cdot \Psi '(0)= t\cdot \mu \\ \mathbb{V}ar X_t = -t \cdot \Psi ''(0)=t\cdot \sigma ^2
\end{align*}
Moreover, for $s\preceq t$, we have
\begin{align*}
\mathbb{E}X_tX_s& = \mathbb{E}(X_t-X_s)X_s + \mathbb{E}(X_s)^2\\
&= \mathbb{E}(X_t-X_s)\mathbb{E}X_s + \mathbb{E}(X_s)^2 \\
&= \mathbb{E}X_{t-s}\mathbb{E}X_s + \mathbb{E}(X_s)^2\\
&= \bigl ( (t-s)\cdot \mu \bigr ) \bigl (s\cdot \mu \bigr ) + s\cdot \sigma ^2 + (s\cdot \mu)^2
\end{align*}
where we used independence and stationarity of the increments along $\mathbb{R}^N_+$ increasing sequences. We thus have
\begin{align*}
Cov(X_t, X_s):= \mathbb{E}X_tX_s - \mathbb{E}X_t \mathbb{E}X_s= s\cdot \sigma ^2
\end{align*}
and the desired result immediately follows.
\end{proof}

\begin{os} \label{decadimento a potenza Levy} Let $|v|$ denote the euclidean norm of $v$.   In the limit  $|t| \to \infty$, we have that $\rho (X_s,X_t)$ behaves like $|t|^{-1/2}$.
Indeed, consider the scalar product in the denominator of \ref{autocorrelazione campi}, i.e. $t\cdot \sigma ^2 = |t|\, |\sigma ^2| \cos \theta$,
where $\theta$ is the angle between $t$ and $\sigma ^2$.
Now, observe that $\sigma ^2$ is a fixed vector of $\mathbb{R}^N_+$, with strictly positive components by the assumption $\psi ''_j(0)\neq 0$. . Since $t$ is in $\mathbb{R}^N_+$ also, by simple geometric arguments it follows that there exist two constants $c_1>0$ and $c_2>0$, which do not depend on $t$, such that $c_1\leq \cos \theta \leq c_2$. Then $k_1 |t|^{-1/2}\leq  \rho (X_s,X_t) \leq k_2   |t|^{-1/2}$ for two suitable constants $k_1>0$ and $k_2>0$ both independent of $t$.
\end{os}

\subsection{Multi-parameter semigroups of operators and their generators}

Let $\mathbb{B}$ be a Banach space equipped with the norm $||\cdot||_{\mathbb{B}}$. A $N$-parameter family $(T_t)_{t\in \mathbb{R}_+^N}$ of bounded linear operators  on $\mathbb{B}$ is said to be a $N$-parameters semigroup of operators if $T_0$ is the identity operator and the following property holds: 
\begin{align}T_{s+t}= T_s \circ T_t  \qquad \forall s,t \in \mathbb{R}_+^N. \label{semigroup property}\end{align}
We say that $(T_t)_{t\in \mathbb{R}_+^N}$ is strongly continuous if 
$$\lim _{t\to 0}||T_t u -u||_{\mathbb{B}}=0 \qquad \forall u\in \mathbb{B}.$$
Moreover, we say that $(T_t)_{t\in \mathbb{R}_+^N}$ is a contraction semigroup if, for any $t\in \mathbb{R}^N_+$, we have $||T_t u||_{\mathbb{B}} \leq ||u||_{\mathbb{B}}$.

\begin{ex}
Let $G_1, G_2, \dots, G_N$ be bounded operators on $\mathbb{B}$, such that $[G_i, G_k]:= G_iG_k-G_kG_i=0$ for all $i\neq k$. Consider the vector
\begin{align*}
G= (G_1, \dots, G_N).
\end{align*}
 Then, for all $t=(t_1, \dots, t_N)$, the family
\begin{align*}
T_t = e^{t_1 G_1} \circ \dots \circ  e^{t_N G_N} = e^{G\cdot t} 
\end{align*}
defines a strongly continuous semigroup on $\mathbb{B}$. In light of the following Definition \ref{definizione generatore}, we will call the vector $G$ the generator of the multiparameter semigroup.
\end{ex}

\begin{ex}
Let $(\mu _t)_{t\in \mathbb{R}^N_+}$ be a multiparameter convolution semigroup of probability measures on $\mathbb{R}^d$ (in the sense of Definition \ref{semigruppo di misure}) and let 
$\mathcal{C}_0(\mathbb{R}^d)$ be the  space of continuous funtions vanishing at infinity, equipped with the sup-norm.
Then 
$$
T_tq(x)= \int _{\mathbb{R}^d } q(x-y) \mu_t(dy)= \mu _t *q(x)  \qquad q\in \mathcal{C}_0(\mathbb{R}^d) \qquad t\in \mathbb{R}^N_+
$$
defines a strongly continuous contraction multi-parameter semigroup.
\end{ex}

Let $t=(t_1, \dots, t_N) \in \mathbb{R}^N_+$ and let $\{e_1, \dots, e_N\}$ be the canonical basis of $\mathbb{R}^N$. For each $j=1, \dots, N$, we refer to the one-parameter semigroups $T_{t_j e_j}$ as the marginal semigroups. By the property \ref{semigroup property} it follows that the marginal semigroups commute, i.e. $[T_{t_i e_i}, T_{t_j e_j} ]=0$ for $i\neq j$ and the following relation holds:
$$ T_t= T_{t_1 e_1}\circ T_{t_2 e_2} \circ \cdots \circ T_{t_Ne_N}$$

Now, let $G_i$ be the generator of $T_{t_i e_i}$, defined on $Dom (G_i)$.       It is well known that if $u\in Dom (G_i)$, then $ T_{t_i e_i} u\in Dom (G_i)$ and the following differential equation
\begin{align*}
\frac{d}{dt_i} w(t_i)= G_i w(t_i) \qquad w(0)= u
\end{align*}
is solved by $w(t_i)=  T_{t_i e_i} u$.
We here report the notion of generator of a multi-parameter semigroup (see [\cite{Butzer}, chapt 1]).

\begin{defin} \label{definizione generatore}
Let $(T_t)_{t\in \mathbb{R}^N_+}$ be a strongly continuous  $N$-parameter semigroup on $\mathbb{B}$ and let $G_i, i=1, \dots, N$, be the generators of the marginal semigroups, each defined on $Dom (G_i)$. We say that the vector
$$ G=(G_1, \dots, G_N)$$
is the generator of $(T_t)_{t\in \mathbb{R}^N_+}$,  defined on $Dom (G)=\bigcap _{j=1}^N Dom (G_j)$ .
\end{defin}

The above definition is intuitively motivated by the following result.

\begin{prop} \label{equazione gradiente}
Let $(T_t)_{t\in \mathbb{R}^N_+}$ be a strongly continuous $N$-parameter semigroup with generator $G$ according to Def.  \ref{definizione generatore}. Then, for $u\in \bigcap _{j=1}^N Dom (G_j)$, the function $w(t)= T_tu$ solves the following system of differential equations
\begin{align}
\nabla _t w(t)= Gw(t) \qquad w(0)=u \label{ppp}
\end{align}
where $\nabla _t$ denotes the gradient with respect to $t= (t_1, \dots, t_N)$.
Namely,  we have
\begin{align}
\frac{\partial}{\partial t_i} w(t)= G_i w(t) \qquad i=1\dots N \label{ogni equazione}
\end{align}
subject to $w(0)=u$.
\end{prop}

\begin{proof}
Let us fix $i=1, \dots, N$. For $q\in Dom (G_i)$ it is true that  $T_{t_ie_i}q \in Dom (G_i)$ and
\begin{align}
\frac{d}{dt_i} T_{t_ie_i}q =G_i T_{t_ie_i}q \label{fff}
\end{align}
By using Propositions 1.1.8 and 1.1.9 in \cite{Butzer}, we know that if $u\in Dom (G_i)$  then $T_tu \in Dom (G_i)$ for any $t\in \mathbb{R}^N_+$. In particular, we have $\bigcirc ^N _{k=1, k\neq i} T_{t_ke_k}u \in Dom (G_i)$  Hence equation \ref{fff} holds for $q= \bigcirc ^N _{k=1, k\neq i} T_{t_ke_k}u $:

\begin{align}
\frac{d}{dt_i} T_{t_ie_i}  \bigcirc ^N _{k=1, k\neq i} T_{t_ke_k}u =G_i T_{t_ie_i} \bigcirc ^N _{k=1, k\neq i} T_{t_ke_k}u 
\end{align}
and the equation \ref{ogni equazione} for a fixed $i$ is found by using property \ref{semigroup property}. By choosing $u\in   \bigcap _{j=1}^N Dom (G_j)$ it is possible to repeat the same argument for all $i=1, \dots, N$, and the system of differential equations is obtained.   
\end{proof}

 By putting $t=0$ in  equation \ref{ppp} it follows that the generator $G$ can also be found by
\begin{align}
Gu= \nabla _t T_tu \bigl |_{t=0}  \qquad u\in    \bigcap _{j=1}^N Dom (G_j)      \label{altra definizione generatore}
\end{align}

For other results concerning multiparameter semigroups and generators consult \cite{Butzer}.
Moreover, for a general discussion on operator semigroups related to multiparameter Markov processes we refer to \cite{khosh}.

\vspace{1cm}
\begin{os}
A different definition of generator for multiparameter semigroups is given in  \cite{jacob} and \cite{schicks}. Here the authors defined the generator as the composition of the marginal generators, i.e.
$$G= G_1\circ  G_2 \circ \dots \circ G_N.$$
The motivation for such definition is that, for $u\in Dom (G_1 \circ \dots \circ G_N)$, the authors prove that $w(t)= T_tu$ solves the partial differential equation
\begin{align} \frac{\partial ^N}{\partial t_1\dots \partial t_N} w(t)=Gw(t)   \qquad    w(0)=u  \label{equazione a derivate parziali} \end{align}
where $t=(t_1, \dots t_N)$.
Also this approach seems to be very interesting, especially in the field of partial differential equations as it allows to find probabilistic solutions to equations of type \ref{equazione a derivate parziali}, containing a mixed derivative.
\end{os}

\subsection{Semigroups associated to Multiparameter L\'evy processes}

Let $(X_t)_{t\in \mathbb{R}^N_+}$ be a Multiparameter L\'evy process on $\mathbb{R}^d$ and let $(\mu _t)_{t\in \mathbb{R}^N_+}$ be the associated convolution semigroup of probability measures, i.e. $\mu _t$ is the law of $X_t$ for each $t$. Consider the operator
\begin{align}
T_t h (x):= \mathbb{E}\,  h(x+X_t)= \int _{\mathbb{R}^d} h(x+y)\mu _t(dy) \qquad h\in  \mathcal{C}_0(\mathbb{R}^d)\qquad t\in \mathbb{R}^N_+ \label{semigruppo levy}
\end{align}
where $\mathcal{C}_0(\mathbb{R}^d)$ denotes the space of continuous functions vanishing at infinity.
By using the properties of  $\{\mu _t\}_{t\in \mathbb{R}^N_+} $ it immediately follows that the family $(T_t)_{t\in \mathbb{R}^N_+}$ is a strongly continuous contraction semigroup on  $\mathcal{C}_0(\mathbb{R}^d)$; it is also positivity preserving, hence it is a Feller semigroup.
We now give a representation of this semigroup and its generator by means of pseudo-differential operators. We restrict to the Schwartz space of functions $\mathcal{S}(\mathbb{R}^d)$.

 We define the Fourier transform by
$$ \hat{h} (\xi)= \frac{1}{(2\pi)^{d/2}} \int _{\mathbb{R}^d} e^{-i\xi \cdot x} h(x)dx \qquad \xi \in \mathbb{R}^d$$
Since $h  \in \mathcal{S}(\mathbb{R}^d)$,  the following Fourier inversion formula holds:
$$
h(x) = \frac{1}{(2\pi)^{d/2}} \int _{\mathbb{R}^d} e^{i\xi \cdot x} \hat{h}(\xi)d\xi \qquad x\in \mathbb{R}^d
$$

\begin{te}\label{semigruppo iniziale} Let $(X_t)_{t\in \mathbb{R}^N_+}$ be a Multiparameter L\'evy process with L\'evy exponent $\Psi$ defined in \ref{funzione caratteristica Levy multiparametrico} and \ref{simbolo Levy multiparametrico}. Let  $(T_t)_{t\in \mathbb{R}^N_+}$ be the associated semigroup defined in \ref{semigruppo levy} and let $G=(G_1, \dots, G_N)$ be its generator. Then
\begin{enumerate}
\item For any $t\in \mathbb{R}^N_+$, $T_t$ is a pseudo-differential operator with symbol $e^{t\cdot \Psi}$, i.e.
\begin{align}T_th(x)=  \frac{1}{(2\pi)^{d/2}} \int _{\mathbb{R}^d} e^{i\xi \cdot x}   e^{t\cdot \Psi (\xi)} \hat{h}(\xi) d\xi \qquad h\in \mathcal{S}(\mathbb{R}^d)  \label{kkk}  \end{align}
\item $G$ is a pseudo-differential operator with symbol $\Psi$, i.e. for each $i=1, \dots, N$ we have
$$ G_ih(x)=  \frac{1}{(2\pi)^{d/2}} \int _{\mathbb{R}^d} e^{i\xi \cdot x}   \psi _i(\xi) \hat{h}(\xi)  d\xi \qquad h\in \mathcal{S}(\mathbb{R}^d)    $$ 

\end{enumerate}

\end{te}

\begin{proof}

\begin{enumerate}
\item Since \ref{semigruppo levy} is a convolution integral, its Fourier transform can be computed as
$$  \frac{1}{(2\pi)^{d/2}} \int _{\mathbb{R}^d} e^{-i\xi \cdot x} T_t h(x)dx= \hat{h}(\xi) \mathbb{E} e^{i\xi\cdot X_t} $$
where $\mathbb{E} e^{i\xi\cdot X_t} =e^{t\cdot \Psi (\xi)} $  by using \ref{funzione caratteristica Levy multiparametrico}. Then Fourier inversion gives the result.

\item By applying formula \ref{altra definizione generatore}, we have that
\begin{align*}G_i h(x)&= \frac{\partial}{\partial t_i} T_tu (x)\biggl |_{t=0}\\
&= \biggl [ \lim _{t_i\to 0}    \frac{1}{(2\pi)^{d/2}} \int _{\mathbb{R}^d} e^{i\xi \cdot x}  \frac{e^{t_i \psi _i(\xi)}-1}{t_i} \prod _{k=1, k\neq i}^N e^{t_k\psi _k(\xi)} \hat{h}(\xi) d\xi  \biggr ]_{t=0}
\end{align*}
The limit can be taken inside the integral due to dominated convergence theorem. Indeed $|  e^{t_k\psi _k(\xi)}    | \leq 1$ for each $k$ because $e^{t_k\psi _k(\xi)}$ is the characteristic function of $\mu _{t_ke_k}$ (see \ref{LLLL}); moreover
$$
\biggl | \frac{e^{t_i \psi _i(\xi)}-1}{t_i} \biggr |\leq |\psi _i(\xi)| \leq C_i (1+|\xi|^2)
$$
where for the last inequality we used [\cite{applebaum} page 31]. Thus the absolute value of the integrand is dominated by $(1+|\xi|^2)\hat{h} (\xi)$. But the last function is independent of $t_i$ and is integrable on $\mathbb{R}^d$ because $\hat{h}$ is a Schwartz function. Then, by exchanging the limit and the integral, the result immediately follows.

\end{enumerate}
\end{proof}

\section{ Composition of random fields  }

\subsection{Subordinator fields} \label{paragrafo subordinator fields}

In order to treat the composition of random fields, the main object is provided by the following definition.

\begin{defin}
A Multiparameter Levy process $(H_{t})_{t\in \mathbb{R}^M_+}$  is said to be a subordinator field if, for some positive integer $N$,  it takes values in $\mathbb{R}^N_+ $ almost surely. 
\end{defin}

The above definition means that, almost surely,  $t\to H_t$   is a non decreasing function with respect to the partial ordering, i.e.  $t_1 \preceq t_2$ on $\mathbb{R}^M_+$ implies $H_{t_1} \preceq H_{t_2}$ on $\mathbb{R}^N_+$.

\begin{ex} \textbf{(Classical subordinators)}
If $N=M=1$, then $(H_{t})_{t\in \mathbb{R_+}}$ is a classical subordinator, i.e. a non-decreasing L\'evy process with values in $\mathbb{R}_+$. Hence it is such that
\begin{align*}
\mathbb{E}e^{-\lambda H_t}= e^{-t f(\lambda)}, \qquad \lambda \geq 0, 
\end{align*}
where the Laplace exponent $f$ is a so-called Bernstein function. Thus it is defined by
\begin{align*}
f(\lambda)= b\lambda + \int _{\mathbb{R}_+} (1-e^{-\lambda x}) \phi (dx)
\end{align*}
 where $b\geq 0$ is the drift cofficient and $\phi$ is the L\'evy measure, which is supported on $\mathbb{R}_+$ and satisfies $\int _{\mathbb{R}_+} \min (x,1) \phi (dx) <\infty$.
For more details on this subject consult \cite{bernstein}.
\end{ex}

\begin{ex} \label{multivariate sub} (\textbf{Multivariate subordinators})

If $M=1$ and  $N\geq 1$, then $(H_t) _{t\in \mathbb{R}_+ }$ is a multivariate subordinator in the sense of \cite{sato}. Thus it is a one-parameter L\'evy process with values  in $\mathbb{R}^N_+$, i.e. it is non decreasing in each marginal component. Here $H_t$ has Laplace transform 
\begin{align*}
\mathbb{E}e^{-\lambda \cdot H_t}= e^{-t S(\lambda)}, \qquad \lambda \in \mathbb{R}^N_+, 
\end{align*}
where the Laplace exponent $S$ is a multivariate Bernstein function. Hence it is defined by
\begin{align*}
S(\lambda)= b\cdot \lambda + \int _{\mathbb{R}^N_+} (1-e^{-\lambda \cdot x}) \phi (dx) \qquad \lambda \in \mathbb{R}^N_+
\end{align*}
where $b \in \mathbb{R}^N_+$, and the L\'evy measure $\phi$ is supported on $\mathbb{R}^N_+$ and satisfies 
$$
\int _{\mathbb{R}^N_+} \min (|x|, 1) \phi (dx) <\infty.
$$
It is known (see e.g. Sect. 2 in  \cite{macci}) that if $H_t$ has a density $p(x,t)$, then it solves 
\begin{align*}
\partial _t p(x,t)= b\cdot  \nabla _x p(x,t)-\mathcal{D}_x p(x,t) \qquad x\in \mathbb{R}^N_+ \qquad t>0
\end{align*}
where $ \mathcal{D}_x$ denotes the $N$-dimensional version of the generalized fractional derivative defined in \ref{prima equazione}, i.e:
\begin{align}
\mathcal{D}_x h(x)= \int _{\mathbb{R}^N_+} \bigl ( h(x)-h(x-y) \bigr ) \phi (dy) \qquad x\in \mathbb{R}^N_+. \label{marchaud multidimensionale}
\end{align}

\end{ex}

\begin{ex}(\textbf{Multivariate stable subordinators}) \label{Multivariate stable subordinators}
We here consider a special sub-case of Example \ref{multivariate sub}, in which the multivariate subordinator is stable.  In order to define this process by means of  its L\'evy measure, we need to use the spherical coordinates $r$ and $\hat{\theta}$, which respectively denote the lenght and the direction of jumps. Clearly $\hat{\theta}$ takes values in the set $\mathcal{C}^{N-1}=\{\hat{\theta}\in \mathbb{R}^N_+: |\hat{\theta}|=1\}$ because, by definition, all the marginal components make positive jumps.     So, a multivariate subordinator $(H_t)_{t\in \mathbb{R}_+}$ is said to be $\alpha$-stable if its L\'evy measure can be written in spherical coordinates as
$$\phi (dr, d \hat{\theta})= \frac{dr}{r^{\alpha +1}} \,  \sigma (d\hat {\theta}) \qquad r>0 \qquad \hat{\theta} \in \mathcal{C}^{N-1}$$
where $\alpha \in (0,1)$ denotes the stability index and $\sigma$ is the so-called spectral measure, which is proportional to the probability distribution of the jump direction $\hat{\theta}$. 
By simple calculations, it is easy to see that in this case the Laplace exponent takes the form
\begin{align}
S^{\alpha, \sigma} (\lambda) = k \int _ { \mathcal{C}^{N-1}}(\lambda \cdot \hat{\theta})^\alpha \sigma (d\hat{\theta}) \qquad \lambda \in \mathbb{R}^N_+
\end{align}
for a suitable $k>0$.
It is known that $H_t$ has a density $p(x,t)$ solving the following equation
\begin{align}
\partial _t p(x,t)= -\mathcal{D}_x^{\alpha, \sigma} p(x,t)
\end{align}
where $\mathcal{D}_x^{\alpha, \sigma}$ is the so-called \textit{fractional gradient}, i.e. a pseudo-differential operator defined by
\begin{align}
\mathcal{D}_x^{\alpha, \sigma} h(x)=  k \int _ { \mathcal{C}^{N-1}}(\nabla \cdot \hat{\theta})^\alpha h(x) \sigma (d\hat{\theta}) \label{fractional gradient}
\end{align}
Note that \ref{fractional gradient} represents the average under $\sigma (d\hat{\theta})$ of the fractional power of the directional
derivative along the direction $\hat{\theta}$. For some theory and applications about this operator consult Example 2.2 in \cite{macci}, chapter 6 in \cite{meer libro} and also \cite{garra, meerschaert benson}.

When $N=2$ the L\'evy measure has the form
$$
\phi (dr, d \theta)= \frac{dr}{r^{\alpha +1}} \,  \sigma (d\theta) \qquad r>0 \qquad    0\leq \theta \leq \frac{\pi}{2}
$$
and, by denoting $\lambda =(\lambda _1, \lambda _2)$, the Laplace exponent can be written as
$$
S^{\alpha, \sigma} (\lambda_1, \lambda _2) = k\int _0^{\pi /2} (\lambda _1 \cos \theta + \lambda _2 \sin \theta) ^\alpha \, \sigma (d\theta),
$$
whence the fractional gradient, acting of a function $(x,y)\to h(x,y)$, has the form
\begin{align}
\mathcal{D}_{x,y}^{\alpha, \sigma} h(x,y)=  k \int _0 ^{\pi /2} \biggl ( \cos \theta \frac{\partial}{\partial x}+  \sin \theta  \frac{\partial}{\partial y} \biggr )^\alpha h(x,y) \sigma (d\theta),
\end{align}
\end{ex}

\subsubsection{The general case}
In the general case where $N$ and $M$ are any positive integers, the Laplace transform of $H_t$ can be computed as follows. Let $t=(t_1, \dots, t_M)\in \mathbb{R}^M_+$ and let $\{e_1, \dots, e_M\}$ be the canonical basis of $\mathbb{R}^M$. We can use Proposition \ref{multiparameter} to say that there exist independent random vectors $Z^{(k)}_{t_k}$, $k=1, \dots, M$,  with $ Z ^{(k)}_{t_k} \overset{d}{=} H_{t_ke_k}$, such that
\begin{align*}
H_t\overset{d}{=} Z^{(1)}_{t_1} + \dots Z ^{(M)}_{t_M}
\end{align*}
But, by the construction of $(H_t)_{t\in \mathbb{R}^M_+}$, it follows that, for each $k=1, \dots, M$, the process  $(H_{t_ke_k})_{t_k \in \mathbb{R}_+}$ is a multivariate subordinator in the sense explained in the previous Example \ref{multivariate sub}. Hence there exist $b_k\in \mathbb{R}^N_+$ and a L\'evy measure $\phi _k$ on $\mathbb{R}^N_+$ (satisfying 
 $\int _{\mathbb{R}^N_+} \min (|x|, 1) \phi _k (dx) <\infty$) such that $H_{t_ke_k}$ has Laplace transform 
\begin{align*}
\mathbb{E}e^{-\lambda \cdot H_{t_ke_k}}= e^{-t_k S_k(\lambda)}, \qquad \lambda \in \mathbb{R}^N_+, 
\end{align*}
where $S_k$ is a multivariate Bernstein functions, defined by
\begin{align}
S_k(\lambda)= b_k\cdot \lambda + \int _{\mathbb{R}^N_+} (1-e^{-\lambda \cdot x}) \phi _k(dx). \label{vvv}
\end{align}

Hence the Laplace transform of $H_t$ can be compactly written as
\begin{align}
\mathbb{E} e^{-\lambda \cdot H_t} = e^{-t_1S_1 (\lambda)\dots - t_MS_M (\lambda) }= e^{-t\cdot S(\lambda)} \label{multidimensional Laplace transform}
\end{align}
where $t=(t_1, \dots, t_M)$ and
\begin{align}
S(\lambda)= \bigl (S_1 (\lambda), \dots, S_M(\lambda) \bigr ) \label{multidimensional Laplace exponent}
\end{align}
We call  \ref{multidimensional Laplace exponent} the multi-dimensional Laplace exponent of the subordinator field.
The above decomposition of a subordinator field into the sum (in distribution) of independent multivariate subordinators will play a decisive role in the following.

A sample path of a stable subordinator field is shown in \autoref{fig:sub-field}.
\begin{figure}
	\includegraphics[width=.6\textwidth]{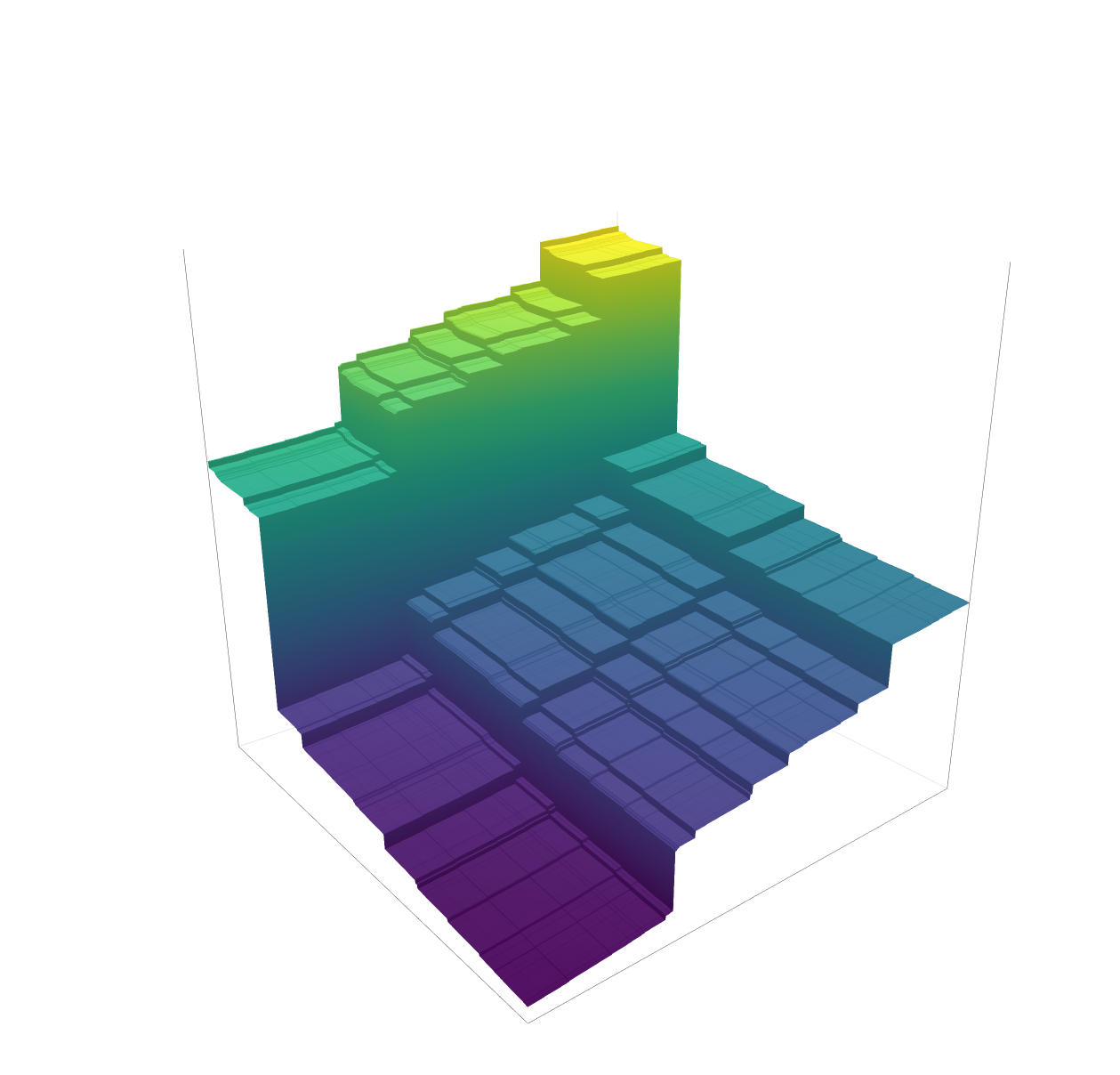}
	\caption{Sample path of a stable subordinator field.}
	\label{fig:sub-field}
\end{figure}

\subsection{Subordinated fields}

Let $(X_{s})_{s\in \mathbb{R}^N_+}$  be a $N$-parameter L\'evy process with values in $\mathbb{R}^d$ and let $(H_{t})_{t\in \mathbb{R}^M_+}$ be a subordinator field (in the sense of Sect. \ref{paragrafo subordinator fields}) with values in $\mathbb{R}^N_+ $.  
In the following, $(X_{s})_{s\in \mathbb{R}^N_+}$ and $(H_{t})_{t\in \mathbb{R}^M_+}$ are assumed to be independent. 
  We consider the subordinated random field
\begin{align}Z_t:=X_{H_t} \qquad t\in \mathbb{R}^M_+. \label{time changed random field}   \end{align}
It is known that  \ref{time changed random field}      is also a Multi-parameter L\'evy process (see [\cite{pedersen 2}, Thm. 3.12]).
Let $\mu _s$, $\rho _t$ and $\nu _t$ respectively denote the probability laws of $X_s$, $H_t$ and $Z_t$.     Then, by conditioning, for any Borel set $B\subset \mathbb{R}^d$ we have
\begin{align}
\nu _t(B)= \int _ {\mathbb{R}^N_+}  \mu _s(B) \, \rho _t(ds). \label{legge processo subordinato}
\end{align}

Processes of type \ref{time changed random field} have also been studied in the literature.

 In \cite{sato} the authors study the case $M=1$ and prove that $(Z_t)_{t\in \mathbb{R}^+}$ is again a L\'evy process and find the characteristic triplet.

 In \cite{pedersen 1}, \cite{pedersen 2} and \cite{pedersen 3}, the authors consider the general case $M\geq 1$; actually their study is more general, since they consider cone-parameter L\'evy processes subordinated by cone-valued L\'evy processes.

\vspace{0.5cm}

Now, let $(T_t) _{t\in \mathbb{R}^N_+}$ be the Feller semigroup associated to $X _t$, defined in \ref{semigruppo levy}, with generator $G=(G_1, \dots, G_N)$.
Moreover, let $(T^Z _t)_{t\in \mathbb{R}^M_+}$ be the Feller semigroup associated to $Z_t$, i.e.
\begin{align}
T^Z _t h(x):= \mathbb{E}\,  h(x+Z_t)= \int _{\mathbb{R}^d} h(x+y) \nu _t(dy) \qquad h\in  \mathcal{C}_0(\mathbb{R}^d) \qquad t\in \mathbb{R}^M_+ \label{semigruppo subordinato z}
\end{align}
where $\nu _t$ is the law of $Z_t$ defined in \ref{legge processo subordinato}, whence we can rewrite \ref{semigruppo subordinato z} as a subordinated semigroup:
\begin{align}
T^Z _t h(x)= \int _{\mathbb{R}^N_+} T_s h(x) \rho _t(ds) \qquad t\in \mathbb{R}^M_+
\end{align}
In the following theorem we determine the form of the generator $G^Z= (G^Z_1, \dots, G^Z_M)$ for the subordinated semigroup, by restricting to the Schwartz space $\mathcal{S}(\mathbb{R}^d)$. We obtain a multiparameter generalization of the well known Phillips formula (see e.g.   [\cite{sato book}, pag. 212]) holding for one-parameter subordinated semigroups.

\begin{te} \label{phillips}
For each $k=1, \dots, M$, we have 
\begin{align}
G^Z _k h(x)= b_k \cdot G\, h(x)+ \int _{\mathbb{R}^N_+} \bigl (T_z    h(x)-h(x) \bigr ) \phi _k (dz) \qquad h\in \mathcal{S}(\mathbb{R}^d).
\end{align}
where $b_k$ and $\phi _k$ have been defined in \ref{vvv}.
\end{te}

\begin{proof}
We first compute the characteristic function of $Z_t= X_{H_t}$.
By conditioning, and using \ref{funzione caratteristica Levy multiparametrico} and \ref{multidimensional Laplace transform}, we have
\begin{align*}
\mathbb{E}e^{i\xi \cdot X_{H_t}}&= \int _{\mathbb{R}^N_+} \mathbb{E}e^{i\xi \cdot X_u} P(H_t \in du)\\ 
 & = \int _{\mathbb{R}^N_+} e^{u\cdot \Psi (\xi)}   P(H_t \in du)\\
&= \mathbb{E}e^{-(-\Psi (\xi)) \cdot H_t}\\
&=e^{-t\cdot S(-\Psi (\xi))} \qquad \xi \in \mathbb{R}^d
\end{align*}
where $t=(t_1, \dots, t_M)$ and 
$$
-S(-\Psi(\xi)):= \begin{pmatrix} -S_1 \bigl ( -\psi _1(\xi), \dots, -\psi _N(\xi) \bigr ) \\ : \\ : \\ -S_M \bigl ( -\psi _1(\xi), \dots, -\psi _N(\xi) \bigr )    \end{pmatrix}
$$

Thus, by using theorem \ref{semigruppo iniziale}, it follows that $T^Z_t$ is a pseudo-differential operator with symbol $e^{-t \cdot S(-\Psi)}$, i.e.

\begin{align}T^Z_th(x)=  \frac{1}{(2\pi)^{d/2}} \int _{\mathbb{R}^d} e^{i\xi \cdot x}   e^{-t \cdot S(-\Psi (\xi))} \hat{h}(\xi) d\xi \qquad h\in \mathcal{S}(\mathbb{R}^d)    \end{align}

while, for each $k=1, \dots, M$,  $G^Z_k$ is a pseudo-differential operator with symbol 
$$-S_k(-\Psi(\xi))=-S_k(-\psi _1 (\xi), \dots, -\psi _N(\xi)).$$
This means that
\begin{align} G^Z_k h(x)=  - \frac{1}{(2\pi)^{d/2}} \int _{\mathbb{R}^d} e^{i\xi \cdot x}   S_k(-\psi _1 (\xi), \dots, -\psi _N(\xi))    \,    \hat{h}(\xi)  d\xi \qquad h\in \mathcal{S}(\mathbb{R}^d)   \label{eee} \end{align}
But, using \ref{vvv}, we have that 
\begin{align}
-S_k(-\Psi (\xi)) = b_k\cdot \Psi (\xi)+ \int _{\mathbb{R}^N_+} (e^{z\cdot \Psi (\xi)}-1) \phi _k(dz) \label{ccc}
\end{align}
Then, after substituting \ref{ccc} in \ref{eee}, we can solve the inverse Fourier transform and taking into account the representation of $T_t$ given in \ref{kkk} we obtain the result.
 
\end{proof}

\begin{os} \label{operational} In the spirit of operational functional calculus, the well known Phillips Theorem (see e.g.  [\cite{sato book}, pag. 212]) can be informally stated as follows. Let a Markov process $(X_t)_{t \in \mathbb{R}_+}$ have generator $G$ and let  a subordinator $(H_t)_{t\in \mathbb{R}_+}$ have Bernstein function $f$. Then the subordinated process $(X_{H_t})_{t\in \mathbb{R}_+}$ has generator $-f(-G)$. 

In a similar way, our Theorem \ref{phillips} can be stated as follows.

 Let $(X_t)_{t\in \mathbb{R}^N_+}$ be a Multiparameter L\'evy process with generator $G=(G_1, \dots, G_N)$ and let $(H_t)_{t\in \mathbb{R}^M_+}$ be a subordinator field associated to the multivariate Bernstein functions $S_1, S_2, \dots, S_M$, namely its Laplace exponent is $S=(S_1, S_2, \dots, S_M)$. Then the subordinated field $(X_{H_t})_{t\in \mathbb{R}^M_+}$ has generator
$$ -S(-G):= \begin{pmatrix} 
 -S_1(-G_1, -G_2, \dots, -G_N)   \\
 -S_2(-G_1, -G_2, \dots, -G_N)   \\
\cdot \\ \cdot \\ 
 -S_M(-G_1, -G_2, \dots, -G_N) 
\end{pmatrix}
$$
\end{os}

\subsection{Stochastic solution to systems of integro-differential equations}

Our extension of the Phillips theorem, given in Theorem \ref{phillips}, provides a stochastic solution to some systems of differential equations. 

Indeed, let $(X_t)_{t\in \mathbb{R}^N_+}$ be a Multiparameter L\'evy process with values in $\mathbb{R}^d$. Moreover, let   $(H_t)_{t\in \mathbb{R}^M_+}$ be a subordinator field with values  in $\mathbb{R}^N_+ $ and let $(Z_t)_{t\in \mathbb{R}^M_+} = (X_{H_t})_{t\in \mathbb{R}^M_+}$ be the subordinated field. Then, by virtue of Proposition \ref{equazione gradiente}, and using the symbolic notation of Remark \ref{operational}, we have that, for any  $u\in \mathcal{S}(\mathbb{R}^d)$, the function $\mathbb{E}u (x+Z_t)$ solves the system 
\begin{align}
\begin{cases} \label{sistema generico}
\frac{\partial}{\partial t_1} h(x,t)= -S_1(-G_1, -G_2, \dots, -G_N)  h(x,t) \\
\frac{\partial}{\partial t_2} h(x,t)= -S_2(-G_1, -G_2, \dots, -G_N)  h(x,t) \\
\cdot \\ \cdot \\ 
\frac{\partial}{\partial t_M} h(x,t)= -S_M(-G_1, -G_2, \dots, -G_N)  h(x,t) \\ \\
h(x,0)= u(x)
\end{cases} \quad x \in \mathbb{R}^d, \,\,\, t\in \mathbb{R}^M_+
\end{align}
where $t= (t_1, \dots, t_M)$,  $G=(G_1, \dots, G_N)$ denotes the generator of  $(X_t)_{t\in \mathbb{R}^N_+}$ and $S_1, \dots, S_M$ are the multivariate Bernstein functions, i.e. the components of the Laplace exponent of  $(H_t)_{t\in \mathbb{R}^M_+}$ defined in     \ref{multidimensional Laplace exponent} .

\begin{ex}
Let $\{e_1, \dots, e_M\}$ be the canonical basis of $\mathbb{R}^M$. Assume that the subordinator field  $(H_t)_{t\in \mathbb{R}^M_+}$ is such that, for each $i=1, \dots, M$, the component $H_{t_ie_i}$ is a multivariate stable subordinator in the sense of Example  \ref{Multivariate stable subordinators}, with index $\alpha _i \in (0,1)$,  whose multivariate Bernstein function reads
\begin{align}
S_i^{\alpha_i, \sigma_i} (\lambda) = k_i \int _ { \mathcal{C}^{N-1}}(\lambda \cdot \hat{\theta})^{\alpha _i} \sigma _i(d\hat{\theta}).
\end{align}
Then the system \ref{sistema generico} takes the form
\begin{align}
\begin{cases}
\frac{\partial}{\partial t_1} h(x,t)= -k_1 \int _{\mathcal{C}^{N-1}} (-G \cdot \hat{\theta} )^{\alpha _1} h(x,t) \sigma _1(d\hat{\theta}) \\
  \frac{\partial}{\partial t_2} h(x,t)= - k_2\int _{\mathcal{C}^{N-1}} (-G \cdot \hat{\theta} )^{\alpha _2} h(x,t) \sigma_2(d\hat{\theta})              \\
\dots \\
  \frac{\partial}{\partial t_M} h(x,t)= -k_M \int _{\mathcal{C}^{N-1}} (-G \cdot \hat{\theta} )^{\alpha _M} h(x,t) \sigma_M(d\hat{\theta}) 
\end{cases} \quad x \in \mathbb{R}^d, \,\,\, t\in \mathbb{R}^M_+
\end{align}
where, on the right side, the fractional powers  $(-G \cdot \hat{\theta} )^{\alpha _i}$ are well defined because $-G \cdot \hat{\theta}$ is the generator of a contraction semigroup.

\end{ex}

\begin{ex} \label{esempio interessante}
Let $N=M=2$. Consider the bi-parameter, additive L\'evy process
\begin{align} X(t_1, t_2)= X_1(t_1)+ X_2(t_2) \label{esempio 2 processo iniziale} \end{align}
\end{ex}
where $X_1$ and $X_2$ are independent isotropic  stable processes with indices $\alpha _1 \in (0,2]$ and $\alpha _2\in (0,2]$ respectively. Let 
$$ H(t_1, t_2)= (H_1(t_1, t_2), H_2 (t_1, t_2))$$ 
be a subordinator field, such that $H(t_1, 0)$ and $H(0, t_2)$ are two bivariate stable subordinators in the sense of Example  \ref{Multivariate stable subordinators} , respectively having indices $\beta _1 \in (0,1)$ and $\beta _2\in (0,1)$ and spectral measures $\sigma _1$ and $\sigma _2$.
Let $$Z(t_1, t_2)= X_1(H_1(t_1, t_2))+ X_2(H_2(t_1, t_2))$$ be the subordinated field. Then, for any  $u\in \mathcal{S}(\mathbb{R}^d)$, the function $\mathbb{E}u(x+Z(t_1, t_2))$ solves the system
\begin{align} \label{secondo esempio}
\begin{cases}
\frac{\partial}{\partial t_1} h(x,t)= - k_1\int _0 ^{\pi/2}  \bigl ( (-\Delta )^{\alpha _1/2}   \cos \theta + (-\Delta )^{\alpha _2/2} \sin \theta    \bigr )^{\beta_1} h(x,t) \sigma _1(d \theta) \\ \\
  \frac{\partial}{\partial t_2} h(x,t)= -k_2 \int _0^{\pi/2} \bigl  ( (-\Delta )^{\alpha _1/2}   \cos \theta + (-\Delta )^{\alpha _2/2} \sin \theta   \bigr  )^{\beta_2} h(x,t) \sigma_2(d \theta) \\ \\
h(x,0)=u(x)             
\end{cases}
\end{align}
where $-(-\Delta )^{\alpha _i/2}$ denotes the fractional Laplacian. To write the system \ref{secondo esempio}, we used that, for $i=1,2$, the generator of the isotropic stable process $X_i$ is     $G_i=-(-\Delta )^{\alpha _i/2}$ (see e.g [\cite{applebaum}, page 166]).

\begin{ex} Consider again Example \ref{esempio interessante}.
In the special case where $\alpha _1= \alpha _2=2$, the process \ref{esempio 2 processo iniziale} is a so-called additive Brownian motion (see e.g. [\cite{khosh}, page 394]) and the above system simplifies to
\begin{align} 
\begin{cases}
\frac{\partial}{\partial t_1} h(x,t)= - C_1 (-\Delta )^{\beta _1}        h(x,t) \\
  \frac{\partial}{\partial t_2} h(x,t)= - C_2 (-\Delta )^{\beta _2}   h(x,t)  \\
h(x,0)=u(x)             
\end{cases}
\end{align}
for suitable constants $C_1, C_2 >0$.
\end{ex}

\begin{ex}
Let $N=1$ and $M>1$ (so that subordination increases the number of parameters).    So let $(X_t)_{t\in \mathbb{R}_+}$ be a one-parameter L\'evy process and let $(H_t)_{t\in \mathbb{R}^M_+}$ be a subordinator field with values in $\mathbb{R}_+$.
For example, assume that  $(X_t)_{t\in \mathbb{R}_+}$ is a standard Brownian motion in $\mathbb{R}^d$ and, for each $k=1, \dots, M$,  $H_{t_ke_k}$ is a stable subordinator of index $\beta _k \in (0,1)$ ($e_k$ denoting the $k$-th vector of the canonical basis).
Let $(Z_t)_{t\in \mathbb{R}^M_+} = (X_{H_t})_{t\in \mathbb{R}^M_+}$ be the subordinated field. Then  $\mathbb{E}u(x+Z_t)$ solves
\begin{align*}
\begin{cases}
\frac{\partial}{\partial t_1} h(x,t)= -(-\Delta )^{\beta _1}  h(x,t) \\
\frac{\partial}{\partial t_2} h(x,t)= - (-\Delta )^{\beta _2} h(x,t) \\
\cdot \\ \cdot \\ 
\frac{\partial}{\partial t_M} h(x,t)= - (-\Delta )^{\beta _M} h(x,t) \\ \\
h(x,0)= u(x)
\end{cases}
\end{align*}

\end{ex}

\section{Subordination by the inverse random field}\label{paragrafo sugli inversi}

Let $(H_t)_{t\in \mathbb{R}_+}$ be a multivariate subordinator in the sense of Example \ref{multivariate sub}, which takes values in $\mathbb{R}^N_+$. Hence it is defined by $H_t=(H_1(t), \dots, H_N(t))$, where each marginal component $H_j(t)$ is a classical subordinator. Consider a new random field $(\mathcal{L}_t)_{t\in \mathbb{R}^N_+}$ defined by
\begin{align} \mathcal{L}_t= \bigl (L_1(t_1), \dots, L_N(t_N) \bigr ) \qquad t=(t_1, \dots, t_N) \label{inverse random field} \end{align}
where  $L_j$ is the inverse hitting time of the subordinator $H_j$, i.e.
$$ L_j(t_j)= \inf \{x>0: H_j(x)>t_j \}$$
As stated in the introduction, we will call \ref{inverse random field} \textit{inverse random field}. 

Now, let $(X_t)_{t\in \mathbb{R}^N_+}$ be a $N$-parameter L\'evy process with values in $\mathbb{R}^d$.
We are interested in the subordinated random field $(Z_t)_{t\in \mathbb{R}^N_+}$ defined by
\begin{align}
Z_t=X_{\mathcal{L}_t} \qquad t\in \mathbb{R}^N_+ \label{campo subordinato con gli inversi}
\end{align}

This topic has many sources of inspiration. Above all, there is a well established theory 
(consult e.g.
\cite{becker,
R6,
kolokoltsov,
R1,
meer nane,
meer spa,
meer jap,
meer straka,
meer libro,
meer toaldo,
straka 2})
concerning semi-Markov  processes of the form
\begin{align} Z(t)= X(L(t)) \qquad t\geq 0 \label{semi markov}\end{align}
where $X$ is a L\'evy process in $\mathbb{R}^d$ and $L$ is the inverse hitting time of a subordinator $H$, i.e.
$$L(t)= \inf \{x>0: H(x)>t  \}$$
Such processes have a great interest in statistical physics, as they arise as scaling limits of suitable continuous time random walks.

\begin{ex} \label{diffusione anomala}
A special case (see e.g.  \cite{R3, R4, magdziarz, magdziarz 2})
is the process
\begin{align} Z(t)=B(L^\alpha (t)) \label{subdiffusion} \end{align}
 where $B$ is a $d$-dimensional standard Brownian motion and $L^\alpha$ is the inverse of a $\alpha$-stable subordinator independent of $B$, where $\alpha \in (0,1)$.   The process \ref{subdiffusion} is a so-called subdiffusion: the mean square displacement behaves as $t^{\alpha}$, i.e. the motion is delayed with respect to the Brownian behavior. This  models the case where the moving particle is trapped by inhomogeneities or perturbations in the medium; thus the particle runs on  Brownian paths, but, for arbitrary time intervals, it is forced to be at rest, which gives rise to a sub-diffusive dynamics. Diffusions in porous media and penetration of a pollutant in the ground have this type of motion (see \cite{metzler} for other applications of anomalous diffusions).  
 The random variable $B(L^\alpha (t))$ has a density solving the following anomalous diffusion equation
\begin{align} \mathcal{D}_t ^\alpha q(x,t)- \frac{t^\alpha}{\Gamma (1-\alpha)}\delta (x)= \frac{1}{2}   \Delta q(x,t) \label{prima eq diffusione} \end{align}
where  $\Delta $ denotes the Laplacian operator and      $ \mathcal{D}_t ^\alpha  $ is the Marchaud fractional derivative, defined by
\begin{align}
 \mathcal{D}_t ^\alpha  h(t):= \int _0^\infty \bigl ( h(t)- h(t-\tau) \bigr ) \frac{\alpha \tau ^{-\alpha -1}}{\Gamma (1-\alpha)} d\tau.
\end{align}
See also \cite{R5} for a tempered version of such operator.
We finally recall that recent models of anomalous diffusion in heterogeneous media, where the fractional order $\alpha$ is space-dependent, have been developed in \cite{R7, ricciuti, mladen} (see also \cite{R2} for a related model).
\end{ex}

Equation \ref{prima eq diffusione} is a special case of a more general theory. Indeed, as anticipated in the Introduction,  if $X$ and $L$ are independent, the connection of the process \ref{semi markov} with integro-differential equations is given by the following facts. Let $X$ have a density $p(x,t)$ solving 
$$\partial _t\, p(x,t)=G^*\, p(x,t)$$
where $G^*$ is the dual to the Markov generator. Moreover, let $L$ be the inverse of a subordinator with L\'evy measure $\nu$.
If $L$ has a density $l(x,t)$, then, by conditioning, $X(L(t))$ has a density 
$$p^*(x,t)= \int _0^\infty p(x,u) l(u,t)du.$$
Such a density solves
\begin{align}\mathcal{D}_t p^*(x,t)  - \overline{\nu} (t) p^*(x,0)    = G^*p^*(x,t)\label{equazione semi markov}\end{align}
where $\overline{\nu}(t)= \int _t^\infty \nu (dx)$ and the operator $\mathcal{D}_t$, usually called generalized Marchaud fractional derivative, is defined by
\begin{align}
 \mathcal{D}_t   h(t):= \int _0^\infty \bigl ( h(t)- h(t-\tau) \bigr )\nu (d\tau). \label{MMM}
\end{align}

Concerning the link between semi-Markov processes and  non-local in time equations, consult also \cite{Michelitsch,pachon} for a discrete-time model and \cite{ricciutitoaldo} for the theory of abstract equations related to semi-Markov Random evolutions.

The rest of this section will be structured as follows. A special case of biparameter L\'evy processes will be treated in subsection \ref{prodotto tensoriale} and a related model of anisotropic subdiffusion will be presented in subsection  \ref{paragrafo diffusione anomala}. Finally, the special case where the $L_j$, $j=1, \dots, N$,  are independent will be presented in subsection  \ref{ultimo paragrafo} and  some long range dependence properties will be analysed.

\subsection{Subordination of some two-parameter L\'evy processes} \label{prodotto tensoriale}

Consider the following biparameter L\'evy process with values in $\mathbb{R}^d$:
\begin{align} X(t_1, t_2)= (X_1 (t_1), X_2(t_2))  \label{bivariate Levy process} \end{align}
where $X_1$ and $X_2$ are (possibly dependent) L\'evy processes with values in $\mathbb{R}^ {d_1}$ and $\mathbb{R}^{d_2}$ respectively, with $d_1+d_2=d$.

Consider now a bivariate subordinator $(H_1 (t), H_2(t))$  and the related bivariate inverse random field $(L_1(t_1), L_2(t_2))$ as defined in \ref{inverse random field}. 

We will consider the following assumptions:

\textit{A1)}  $X_1 (t_1)$ and $X_2(t_2)$ have marginal densities $p_1(x_1,t)$ and $p_2(x_2,t)$
satysfying the following forward equations:
$$ \frac{\partial}{\partial t} p_i(x_i, t)= G^*_i \, p_i(x_i,t) \qquad i=1,2$$
where $G^*_1$ and $G^*_2$ are the duals to the generators of $X_1$ and $X_2$.

\vspace{0.1cm}

\textit{A2)} $X(t_1, t_2)$ has density $p(x_1, x_2, t_1, t_2)$ satysfying the system 
$$ \frac{\partial}{\partial t_i} p(x_1, x_2, t_1, t_2)= G_i^* p(x_1, x_2, t_1, t_2)  \qquad i=1,2.$$

\vspace{0.1cm}

\textit{A3) } For all $t_1, t_2>0$, the random vector $(H_1(t_1), H_2(t_2))$ has a density $q(x_1, x_2, t_1, t_2)$
\footnote{Observe that the random field   $(t_1, t_2) \to (H_1(t_1), H_2(t_2))$  is not a biparameter L\'evy process even if $t\to (H_1(t), H_2(t))$  is a multivariate subordinator, unless the two marginal components are independent. }.
\vspace{0.1cm}

 We now consider the subordinated random field
\begin{align} 
	Z(t_1, t_2)= X(L_1(t_1), L_2(t_2)) 
	\label{caso prodotto} 
\end{align}

The following Proposition gives a generalization of equation \ref{equazione semi markov} adapted to the random field \ref{caso prodotto}.

\begin{prop} \label{equazione inversi}
Under the assumptions A1),   A2) ,  A3) , the random vector  $X(L_1(t_1), L_2(t_2))$ has a density $h(x_1, x_2, t_1, t_2)$ satisfying 
\begin{align}
\mathcal{D}_{t_1, t_2} h (x_1, x_2, t_1, t_2) = (G^*_1+G^*_2) h (x_1, x_2, t_1, t_2) \qquad x_1\neq 0, x_2 \neq 0
\end{align}
where $\mathcal{D}_{t_1, t_2}$ is the bidimensional version of the generalized fractional derivative, defined in \ref{marchaud multidimensionale}, i.e.
$$
\mathcal{D}_{t_1, t_2} h(t_1, t_2)= \int _{\mathbb{R}^2_+} \bigl ( h(t_1, t_2)- h(t_1-\tau _1, t_2 -\tau _2) \bigr ) \phi (d\tau _1, d\tau _2)
$$
\end{prop}

\begin{proof}

Under assumption \textit{A3)},  the distribution of $(L_1(t_1), L_2(t_2))$ is the sum of two components (see [\cite{macci}, sect. 3.1]): the first one is absolutely continuous with respect to the bi-dimensional Lebesgue measure, with density $l$, namely
$$ P(L_1(t_1) \in dx_1, L_2(t_2) \in dx_2) = l(x_1, x_2, t_1, t_2) dx_1 dx_2 \qquad x_1\neq x_2$$
while the second one has support on the bisector line $x_1 = x_2$, with one dimensional Lebesgue density $l_*(x,t_1, t_2)$ (i.e.  $P(L_1(t_1)= L_2(t_2))= \int _0^\infty l_*(x,t_1, t_2)dx)$.

Then, by using a simple conditioning argument,  the random vector $X(L_1(t_1), L_2(t_2))$ has density
\begin{align*} h(x_1, x_2, t_1, t_2)= &\int _0^\infty \int _0^\infty p(x_1, x_2, u,v) l(u,v,t_1, t_2) dudv \\ + & \int _0^\infty p(x_1, x_2, u,u) l_*(u, t_1, t_2)du \end{align*}

By applying $\mathcal{D}_{t_1, t_2}$ to both sides and using [\cite{macci}, Thm 3.6] we have
\begin{align}
& \mathcal{D}_{t_1, t_2}h(x_1, x_2,t_1, t_2)  \notag \\
&=- \int _0^\infty \int _0^\infty p(x_1, x_2,u,v) \frac{\partial}{\partial u}
l(u,v,t_1, t_2)dudv  \notag \\
&- \int _0^\infty \int _0^\infty p(x_1, x_2,u,v) \frac{\partial}{\partial v}
l(u,v,t_1, t_2)dudv  \notag \\
& - \int _0^\infty  p(x_1, x_2, u,u)   \frac{\partial}{\partial u}l_*(u,t_1,
t_2) du .
\end{align}
Now, we integrate by parts by using  assumptions \textit{A1} and  \textit{A2}. We also use that $X_1(0)=0$ and $X_2(0)=0$ almost surely, which implies that $P(X_1(0) \in A, X_2(t_2)\in B)= \mathcal{I} _{(0\in A)} P(X_2 (t_2)\in B)$ and $P(X_1(t_1) \in A, X_2(0)\in B)= P(X_1 (t_1)\in A) \mathcal{I} _{(0\in B)}$; thus we get
\begin{align*}
& \mathcal{D}_{t_1, t_2}h(x_1, x_2,t_1, t_2) \\
&= G^*_1 \int _0^\infty \int _0^\infty p(x_1, x_2, u,v) l(u,v,t_1,
t_2)dudv +\delta (x_1) \int _0^\infty p_2(x_2, v) l(0,v,t_1t_2)dv + \\
&+ G^*_2 \int _0^\infty \int _0^\infty p(x_1, x_2, u,v) l(u,v,t_1,
t_2)dudv +\delta (x_2) \int _0^\infty p_1(x_1, u) l(u,0,t_1t_2) du \\
& + (G^*_1+G^*_2) \int _0^\infty p(x_1, x_2, u,u) l_*(u,t_1,
t_2) du +\delta (x_1)\delta (x_2) \overline{\phi}(t_1, t_2)
\end{align*}
where $$\overline{\phi}(t_1, t_2)= \int _{t_1}^\infty \int _{t_2}^\infty \phi (dx_1, dx_2).$$
In the above calculations we have taken into account that  $$\frac{\partial p(x_1, x_2, u,u)}{\partial u}= (G^*_1 +G^*_2)p(x_1, x_2, u,u)$$ since the total derivative of $p(x_1, x_2, t_1, t_2)$, with $t_1=u$ and $t_2=u$, is given by 
$$ \frac{\partial p}{\partial t_1} \frac{\partial t_1}{\partial u}+   \frac{\partial p}{\partial t_2} \frac{\partial t_2}{\partial u}= G^*_1 p+ G^*_2p.$$

 In the region $x_1\neq 0, x_2 \neq 0$ we have
\begin{align*}
& \mathcal{D}_{t_1, t_2}h(x_1, x_2,t_1, t_2) \\
=& (G^*_1 +G^*_2) \int _0^\infty \int _0^\infty  p(x_1, x_2, u,v) l(u,v,t_1, t_2)dudv \\
& + (G^*_1+G^*_2) \int _0^\infty  p(x_1, x_2, u,u) l_*(u,t_1,
t_2) du \\
& = (G^*_1 +G^*_2) h(x_1, x_2, t_1, t_2),
\end{align*}
which concludes the proof.
\end{proof}

A sample path of a time-changed field is shown in \autoref{fig:tchange-field}.
\begin{figure}
	\includegraphics[width=0.7\textwidth]{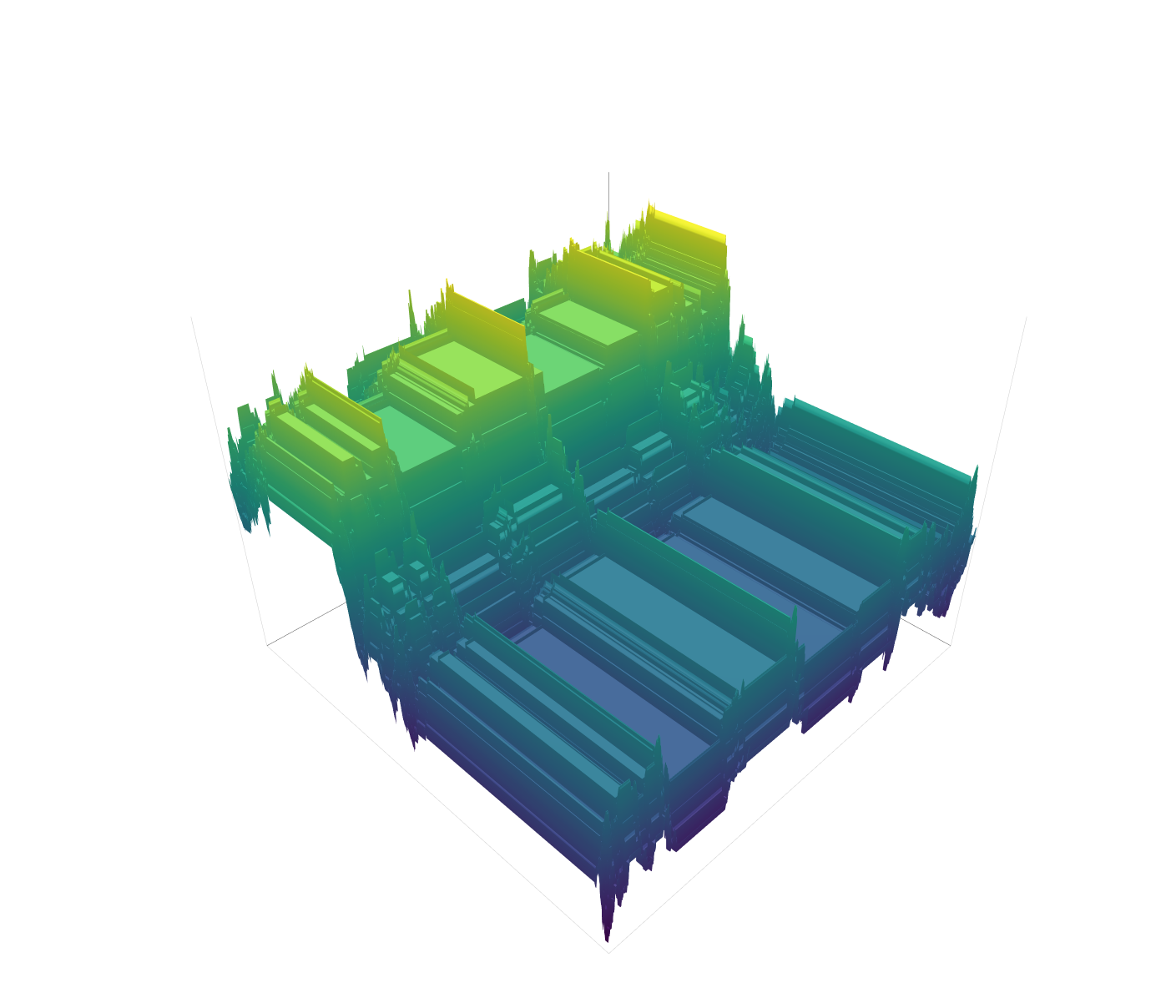}
	\caption{Sample path of a time-changed additive Brownian field with an inverse stable field.}
	\label{fig:tchange-field}
\end{figure}

\subsection{Anomalous diffusion in anisotropic media} \label{paragrafo diffusione anomala}

As a byproduct of the results of section \ref{prodotto tensoriale},   we here propose another model of  subdiffusion which  extends the one treated in Example \ref{diffusione anomala}, by including it as a special case.

As explained, the process \ref{subdiffusion} models a subdiffusion through an isotropic medium, i.e. the trapping effect is the same in all coordinate directions (e.g. all components of the Brownian motion are delayed by the same random time process). Hence  the subordinated process \ref{subdiffusion} is isotropic as well as the Brownian motion. 

Thus it is natural to search for a model of subdiffusion in the case where the external medium is not isotropic. Actually, a first model of anisotropic subdiffusion has been proposed in [\cite{macci}, Sect. 5].  In the following, we will improve such a model, by including it in a more general framework.

We recall some notions on operator stability (consult \cite{jurek} and \cite{sharpe}). A random vector $X$ with values in $\mathbb{R}^d$ is said to be \textit{operator stable} if, for any positive integer $n$, there exist a vector $c_n\in \mathbb{R}^d$ and a $d\times d$ matrix $A$ such that $n$ indipendent copies $X_1, \dots, X_n$ of $X$ satisfy
\begin{align} X_1+\dots + X_n \overset{d}{=} n^AX+c_n \label{operator stability} \end{align}
where the matrix power $n^A$ is defined by 
$$n^A= e^{A\ln n}= \sum _{k=0}^\infty \frac{1}{k!} A^k (\ln n)^k. $$ 
In the special case $A=\frac{1}{\alpha} I$, with $\alpha \in (0,2]$ and $I$ denoting the identity matrix,  we have that $X$ is $\alpha$-stable. In the general case, $A$ has eigenvalues whose real parts have the form $1/\alpha _i$, with $\alpha _i \in (0,2], i=1, \dots, d$. We stress that the matrix $A$ is not unique, i.e. there may be different $n\times n$ matrices satisfying \ref{operator stability} (unlike what happens in the stable case, where the index $\alpha$ is uniquely defined).

Operator stable laws are infinite divisible, hence they correspond to some L\'evy processes. A Levy process $X(t), t\geq 0$ is said to be an operator stable L\'evy motion if $X(1)$ is an operator stable random vector.  Note that such a process is characterized by the anisotropic scaling $X(ct) \overset{d}{=}c^A X(t)$.  This property is  a generalization of self-similarity of $\alpha$-stable processes where the scaling is the same for all coordinates, i.e.  $X(ct) \overset{d}{=} c^{1/\alpha} X(t) $.

We are now ready to present the model of anisotropic subdiffusion.
So, let us consider a bivariate subordinator $(H_1(t), H_2(t))$   which is constructed as an operator stable L\'evy motion with values in $\mathbb{R}^2_+$. In this case $A$ has eigenvalues whose real parts have the form $1/\alpha _i$, with $\alpha _i \in (0,1)$, $i=1,2$. Now, let $r>0$ and $ \theta \in [0, \frac{\pi}{2}]$ be the so-called Jurek coordinates (see e.g. \cite{jurek} and [\cite{meer libro}, page 185]) which are defined by the mapping $\mathbb{R}^2_+\ni x= r^A\hat{\theta}  $, where  $\hat{\theta}= (\cos \theta, \sin \theta)$.
 In this new coordinates the bi-dimensional L\'evy measure can be expressed as
$$ \phi ^{A,M} (dr, d\theta)= C \frac{dr}{r^2}M(d\theta) \qquad r>0 \qquad \theta \in \biggl [0, \frac{\pi}{2} \biggr ]$$
where $M$ is a probability measure on the angular component.  Then the operator $\mathcal{D}_{x}$, $x \in \mathbb{R}^2_+$, defined in formula \ref{marchaud multidimensionale} of Example \ref{multivariate sub}, takes the form
\begin{align} \mathcal{D}_{x}^{A,M} h(x)= C \int _0^{\pi/2}\int _{0}^\infty  \bigl (h(x)-h(x-r^A\hat{\theta}) \bigr )   \frac{dr}{r^2}M(d\theta) \label{AM}\end{align}
If $(H_1(t), H_2(t))$ is a bivariate stable subordinator (see Example \ref{Multivariate stable subordinators}), i.e. $A= \frac{1}{\alpha} I $, by a simple change of variables  one re-obtains the fractional gradient defined in formula \ref{fractional gradient}.

  Now, let $(L_1(t_1), L_2(t_2))$ be the inverse random field of $(H_1(t), H_2(t))$ and let $(B_1(t), B_2(t))$ be a bi-dimensional standard Brownian motion with independent components. Consider the time changed process 
\begin{align}Z(t)= \bigl(B_1(L_1(t)), B_2(L_2(t)) \bigr) \qquad t \geq 0 \label{anisotropic subdiffusion} \end{align}

The process \ref{anisotropic subdiffusion} is a model of anisotropic subdiffusion. Indeed consider the random variable
$$Z_{\theta} (t)= Z(t) \cdot \hat{\theta}$$
representing the displacement along the direction $\hat{\theta} = (\cos \theta, \sin \theta)$. By conditioning, the mean square displacement can be written as
\begin{align*}
\mathbb{E}Z_{\theta} ^2 (t) 
= \mathbb{E} L_1(t) \cos ^2 \theta + \mathbb{E} L_2 (t) \sin ^2 \theta
\end{align*}
which, in general, depends on $\theta$ because of anisotropy.

In the spirit of  [\cite{macci}, Sect. 4], a governing equation for the process \ref{anisotropic subdiffusion} can be obtained by considering the related random field $(B_1(L_1(t_1)), B_2(L_2(t_2)))$. Indeed, by applying Proposition \ref{equazione inversi} of the previous section, it has a density $h(x_1, x_2, t_1, t_2)$ satisfying 
the anomalous diffusion equation 
$$
\mathcal{D} ^{A,M} _{t} h(x_1, x_2, t_1, t_2)= \frac{1}{2} \Delta \, h(x_1, x_2, t_1, t_2) \qquad x_1 \neq 0, x_2 \neq 0 
$$
where the operator $\mathcal{D} ^{A,M} _{t}$, defined in \ref{AM}, now acts on  $t=(t_1, t_2)$.

\begin{ex}
If $L_1(t)=L_2(t)=L(t)$, where $L(t)$   is the inverse of a $\alpha$-stable subordinator, the process \ref{anisotropic subdiffusion} reduces to the isotropic subdiffusion \ref{subdiffusion}. In this case we have $\mathbb{E}L(t)=Ct^\alpha$. Thus   $\mathbb{E}Z_{\theta} ^2 (t)=Ct^\alpha$, which is independent of $\theta$ because of isotropy.
\end{ex}

\begin{ex} \label{esempio 2}
If $H_1(t)$ and $H_2(t)$ are independent stable subordinators, then the matrix $A$ is diagonal with elements $1/\alpha _1$ and $1/\alpha _2$. If $\alpha _1 \neq \alpha _2$ the  process \ref{anisotropic subdiffusion} is anisotropic, in such a way that $\alpha _1$ and $\alpha _2$ represent the spreading rates along the two coordinate directions. Indeed,  since $\mathbb{E}L_i(t)= C_i t^ {\alpha _i}$ for $i=1,2$, then the mean square displacement  along a direction $\hat{\theta}$ has the form $\mathbb{E} Z_\theta ^2 (t) =  C_1 t^ {\alpha _1} \cos ^2 \theta +  C_2 t^ {\alpha _2} \sin ^2\theta $ which depends on $\hat{\theta}$ and  asymptotically behaves like $t ^{\max (\alpha _1, \alpha _2)}$.
\end{ex}

\begin{ex}
If $A$ is a symmetric matrix with eigenvalues $1/\alpha _1$ and $1/\alpha _2$, where $\alpha _1$ and $\alpha _2$ are in $(0,1)$, then a rigid rotation of the coordinate system allows to find the two eigenvectors, along which the spreading rates are $\alpha _1$ and $\alpha _2$ respectively, which corresponds to the situation explained in Example \ref{esempio 2} .
\end{ex}

\subsection{Subordination by independent inverses } \label{ultimo paragrafo}

In the following, let $X(t_1, \dots, t_N)$ be a $N$-parameter L\'evy process with density $p(x,t)$ satisfying the system
\begin{align*}
\partial _{t_j} p(x,t)= G^*_j p(x,t) \qquad j=1, \dots, N
\end{align*}
with the usual notation $t=(t_1, \dots, t_N)$.  Assume that the marginal components $L_j(t_j)$ of the inverse random field \ref{inverse random field} are mutually independent, each having density $l_j(x, t_j)$ and L\'evy measure $\nu _j$.   Consider the subordinated random field
\begin{align}Z(t):=X(L_1(t_1), \dots, L_N(t_N)) \label{nuovo campo}
\end{align}
Before stating the next result, we introduce the following notation: for a given vector $v= (v_1, \dots, v_N)$, we introduce the vector $v^{(j)}$ defined by    $v^{(j)}= (v_1, \dots, v_{j-1}, 0, v_{j+1}, \dots, v_N)$.

\begin{prop}
Under the above assumptions, the subordinated field \ref{nuovo campo} has a density
 $p^*(x,t)$ satisfying the system
\begin{align*}
\mathcal{D}^{(\nu _j)} _{t_j}\,  p^* (x,t) - \overline {\nu}_j(t_j) \, p^{*} (x, t^{(j)})  = G_j ^*\, p^* (x,t) \qquad j=1, \dots, N
\end{align*}
where $\mathcal{D}^{(\nu _j)} _{t_j}$ denotes the generalized fractional derivative defined in \ref{MMM} with L\'evy measure $\nu _j$, and $\overline {\nu}_j(t_j)= \int _{t_j}^\infty \nu _j(d\tau)$.
\end{prop}

\begin{proof}
By conditioning, \ref{nuovo campo} has a density
$$p^{*}(x,t)= \int _{R ^N _+} p(x, u_1, \dots, u_N)    \prod_{i=1 }^N l_i(u_i, t_i)        du_1 \cdots du_N$$
By applying $\mathcal{D}^{(\nu _j)} _{t_j}$ to both members and taking into account that such operator commutes with the integral, we have
$$ 
\mathcal{D}^{(\nu _j)} _{t_j} p^{*}(x,t)=-\int _{R ^N _+} p(x, u_1, \dots, u_N) \frac{\partial}{\partial u _j} l_j(u_j,t_j)  \prod_{i=1,  i\neq j }^N l_i(u_i, t_i)\, du_1 \cdots du_N
$$
where we used that the density $l_j(x,t_j)$ of an inverse subordinator satisfies the equation $\mathcal{D}^{(\nu _j)} _{t_j}  l_j(x,t_j)=-     \partial _x l_j(x,t_j)$ under the condition $l_j(0, t_j)= \overline{\nu}_j(t_j)$ (see e.g. \cite{kolokoltsov}).

Integrating by parts, we have
$$ 
\mathcal{D}^{(\nu _j)} _{t_j} p^{*}(x,t)= G_j^*\,  p^*(x,t)+\overline{\nu}_j(t_j) \int _{\mathbb{R}^{N-1}_+} p(x, u^{(j)})  \prod_{i=1,  i\neq j }^N l_i(u_i, t_i)du_i 
$$
where the last integral can be written as
$$
  \int _{\mathbb{R}^{N-1}_+} p(x, u^{(j)})  \prod_{i=1,  i\neq j }^N l_i(u_i, t_i)du_i =p^*(x, t^{(j)}) 
$$
because $l_j(u_j,0)= \delta (u_j)$. This completes the proof.
\end{proof}

\subsubsection{Long range dependence.}

Consider a process of type \ref{nuovo campo}. For each $k=1, \dots, N$, let  $L_k(t_k)$ be the inverse of a $\alpha$-stable subordinator.
The subordinated field exhibits a power law decay of the auto-correlation function which is slower with respect to the $|t|^{-\frac{1}{2}}$ decay holding for Multiparameter L\'evy processes (which was discussed in Remark  \ref{decadimento a potenza Levy} ). This can be useful in applied fields, where spatial data exhibit  long range dependence properties.

So, let $s\preceq t$. By using the results of Section \ref{paragrafo autocorrelazione},      we have
\begin{align*}
& Cov  (X_{L_s}, X_{L_t} ) \\
&= \mathbb{E} \bigl [  Cov  (X_{L_s}, X_{L_t} ) \bigl  | \, L_s, L_t  \bigr ] + Cov \bigl ( \mathbb{E}[ X_{L_s} | \, L_s, L_t  ],  \mathbb{E}[ X_{L_t} | \, L_s, L_t  ]     \bigr ) \\
& = \mathbb{E} [L_s \cdot \sigma  ^2] + Cov ( L_t \cdot \mu, L_s \cdot \mu )\\
&=  \mathbb{E} \biggl [  \sum _{k=1}^N \sigma _k^2 L_k(s_k)   \biggr  ] + Cov \biggl ( \sum _{k=1} ^N \mu _k L_k(t_k)      ,    \sum _{i=1} ^N \mu _i L_i(s_i)   \biggr )\\ 
& = \sum _{k=1}^N \sigma _k^2 \mathbb{E}L_k(s_k) +  \sum _{k=1} ^N  \sum _{i=1} ^N   \mu _k   \mu _i \, Cov \bigl ( L_k(t_k)      ,   L_i(s_i)  \bigr ) \\
&=  \sum _{k=1}^N \sigma _k^2 \mathbb{E}L_k(s_k) + \sum _{k=1}^N    \mu _k^2 \,   Cov \bigl ( L_k(t_k)      ,   L_k(s_k)  \bigr )
\end{align*}
where in the last step we used independence between  $L_i$ and $L_k$ when $i\neq k$. 
Putting $s=t$ we have
\begin{align*}
\mathbb{V}ar  X_{L_t} =  \sum _{k=1}^N \sigma _k^2 \mathbb{E}L_k(t_k) + \sum _{k=1}^N    \mu _k^2 \, \mathbb{V}ar L_k(t_k)     
\end{align*}

 By self-similarity of the inverse stable subordinator (consult e.g. Proposition 3.1 in \cite{meer jap}), we have $L_k(t_k)  \overset{d}{=} t_k ^\alpha L_k(1)$. Hence 
$$\mathbb{E}L_k(t_k)  = t_k ^\alpha \, \mathbb{E} L_k(1)  \qquad    \mathbb{V}ar L_k(t_k)  = t_k ^{2\alpha} \, \mathbb{V} ar L_k(1).$$

Thus, by using the notation $t^\beta := (t_1 ^\beta, \dots, t_N^\beta)$ we can write 
\begin{align*}
\mathbb{V}ar  X_{L_t} =  w\cdot t^\alpha + v\cdot t^{2\alpha} 
\end{align*}
where we defined  $w_k= \sigma _k^2\, \mathbb{E}L_k(1)$ and $v_k= \mu _k^2 \, \mathbb{V}ar L_k(1)$.

Moreover, by using Formula 10 in \cite{Leonenko 2} we have 
$$Cov \bigl ( L_k(t_k)      ,   L_k(s_k)  \bigr ) \sim \frac{s_k^{2\alpha}}{\Gamma (2\alpha +1)} \qquad t_k\to \infty.$$

In summary, for $| t| \to \infty$, we have
\begin{align} \label{autocorrelazione subordinato}
\rho (X_{L_s}, X_{L_t}) \sim  \begin{cases}    \frac{1}{|t^\alpha| ^{1/2}} \qquad \qquad &\textrm{if} \,\, \mu =0 \\  \frac{1}{|t^{2\alpha}|^{1/2}   }  \qquad & \textrm{if}\,\, \mu \neq 0 \end{cases}
\end{align}

\begin{os}
 What we found in   \ref{autocorrelazione subordinato}  is the multiparameter extension of the known formula holding in the $N=1$ case, see e.g.  Example 3.2 in \cite{Leonenko 2}. Here the authors considered the subordinated process $(X_{L(t)})_{t\in \mathbb{R}_+}$,  where $(X_t) _{t\in \mathbb{R}_+}  $ is  a L\'evy process and $(L(t))_{t\in \mathbb{R}_+}    $ is the inverse of a $\alpha$-stable subordinator, with $\alpha \in (0,1)$.
By considering two times $s$ and $t$, such that $s<t$, and letting $t\to \infty$,  they show that the auto-correlation $\rho (X_{L(t)}, X_{L(s)})$ behaves like $t^{-\alpha}$ if $\mathbb{E}X_1\neq 0$ and $t^{-\frac{\alpha}{2}}$ if $\mathbb{E}X_1= 0$.  It is interesting to note that the same power law behavior is observed in the corresponding discrete-time models (see Proposition 4 in \cite{pachon}).
\end{os}

\end{document}